\documentclass[oneside,11pt]{amsart}

\newskip\stdskip                      
\usepackage{pinlabel}  
\usepackage{pictexwd,dcpic,xypic}
\usepackage{amscd,latexsym,amssymb}
\usepackage{bm}
  \usepackage{hyperref}  

   \title{On the graph products of simplicial groups and connected Hopf algebras}
\author[L. Cai]{Li Cai}
\address{Department of Pure Mathematics, Xi'an Jiaotong-Liverpool University, 111 Ren'ai Road, Dushu Lake Higher Education Town, Suzhou 215123, Jiangsu, China}
\email{Li.Cai@xjtlu.edu.cn}
\thanks{Supported in part by National Natural Science Foundation of China (grant no. 11801457). The author would like to thank the referee for the valuable suggestions and comments which greatly improve this article.}

\newtheorem{thm}{Theorem}[section]

\newtheorem{cor}[thm]{Corollary}
\newtheorem{lem}[thm]{Lemma}
\newtheorem{prop}[thm]{Proposition}

\theoremstyle{definition}
\newtheorem{defin}[thm]{Definition}

\theoremstyle{definition}

\newtheorem{exm}[thm]{Example}

\newtheorem{remark}[thm]{Remark}

\theoremstyle{remark}

\def\co{\colon\thinspace}
\def\wt{\widetilde}

\begin{document}

\begin{abstract}    
In this paper we study the classifying spaces of graph products of simplicial groups and connected Hopf algebras over a field, and show that they can be uniformly treated under the framework of polyhedral products. It turns out that these graph products are models of the loop spaces of polyhedral products over a flag complex and their homology, respectively. Certain morphisms between graph products are also considered. In the end we prove the structure theorems of such graph products in the form we need.   
\end{abstract}

\maketitle

\section{Introduction}
By the Seifert-van Kampen theorem, the wedge sum $K(G_1,1)\vee K(G_2,1)$ of classifying spaces of discrete groups $G_1$ and $G_2$ is again the classifying space $K(G_1*G_2,1)$ of the free product $G_1*G_2$. However, the situation in higher dimensions is much more complicated. For instance, 
for $K(\mathbb{Z},2)\vee K(\mathbb{Z},2)$  we have a homotopy equivalence $\Omega \left(K(\mathbb{Z},2)\vee K(\mathbb{Z},2)\right)\simeq \Omega S^3\times S^1\times S^1$ of Moore loop spaces, by Ganea \cite{Gan65}. Together with a result of Selick \cite{Sel77}, for every prime $p>3$, the $p$-primary torsion parts in the homotopy group $\pi_*(K(\mathbb{Z},2)\vee K(\mathbb{Z},2))$ are (infinitely many copies of) $\mathbb{Z}/p$.

Polyhedral products in the sense of \cite{BBCG10} give a much larger class of spaces, including the wedge sums and 
cartesian products, as well as an important class of manifolds called moment-angle manifolds, which is the main class of compact complex manifolds which are not K\"{a}hler (see \cite{BM06}). The homotopy decomposition of loop spaces of polyhedral products have been well studied (see \cite{The18}, \cite{PT19}). 

In this work we consider the classifying space $\overline{W}(\underline{G}^K)$ of the graph product $\underline{G}^K$ of simplicial groups $G_1,\ldots, G_m$, where $K$ is the flag complex obtained by filling all triangles and simplices of higher dimensions from the graph. It turns out that $\overline{W}(\underline{G}^K)$ is homotopy equivalent to the polyhedral product $\underline{\overline{W}(G)}^K$ of  $\overline{W}(G_1),\ldots, \overline{W}(G_m)$ with respect to $K$ (see Theorem \ref{thm:weakequi}).\footnote{Here $\underline{X}^K$ is a simplified notation of $\underline{(X,*)}^K$. See Section \ref{sec:owa} for more details.} This is a generalization of a result of Stafa \cite{Sta15} (see Corollary \ref{Cor:Sta15}). Likewise, in the category $\mathcal{H}_k^c$ of connected Hopf algebras over a field $k$, the bar construction $\bar{B}(\underline{A}^K)$ of the graph product of objects $A_1,\ldots, A_m$ in $\mathcal{H}_k^c$ is given by the polyhedral product $\underline{\bar{B}(A)}^K$ of simplicial vector spaces $\bar{B}(A_i)$, $i=1,\ldots, m$ (see Theorem \ref{thm:isophi}).  

Let $k\underline{G}^K$ be the simplicial group ring over $k$ such that the face and degeneracy maps are $k$-linear. 
As a result, we have the isomorphism $H_*(k\underline{G}^K)\cong\underline{A}^K$ in $\mathcal{H}_k^c$, with $A_i=H_*(kG_i)$ where $G_i$ is reduced, $i=1,\ldots, m$. When $G_i=G(X_i)$ is the Kan construction on a simply connected $CW$ complex $X_i$, $i=1,\ldots,m$, we get a result of  Dobrinskaya \cite{Dob09} on the Pontryagin algebra of the loop space $\Omega\underline{X}^K$. 

Let $H$ be the kernel of the morphism $\underline{G}^K\to\prod_{i=1}^mG_i$ of simplicial groups (the abelianization with respect to $K$). Using Milnor's construction \cite{Mil72} which generalizes Hilton's result on the homotopy type of the loop space of a wedge sum of spheres, we give a topological interpretation of a result of Panov and Veryovkin \cite{PV19} on the generators of $H$ (see Corollary \ref{cor:PV}). We believe that, at least in some cases, this is dual to the inclusion of bottom cells in the homotopy fiber $F$ of the inclusion $|\underline{\overline{W}(G)}^K|\to\prod_{i=1}^m|\overline{W}(G_i)|$ (see Example \ref{exm:Medrano}). Motivated by the Pontryagin algebra of the loop space of a moment-angle complex (see \cite{GPTW16}), in the second part we consider $A'=H_*(kH)$ (with $G_i$ reduced, $i=1,\ldots,m$).   Working in the category $\mathcal{H}_k^c$, $A'=\underline{A}^K\square_{\otimes_{i=1}^m A_i}k$ is the kernel of the abelianization $\underline{A}^K\to\otimes_{i=1}^mA_i$, and we prove an isomorphism
\[
Tor^{A'}_{s,*}(k,k)\cong \bigoplus_{I\subset\{1,\ldots,m\}}\wt{H}_{s-1}(K_{I};k)\otimes\otimes_{i\in I}J(A_i)\] 
with $s\geq 1$, $K_I$ the corresponding full subcomplex and $J(A_i)=\mathrm{ker}(A_i\to k)$ the augmentation ideal. This is a generalization of a theorem of Vylegzhanin \cite{Vyl22}. The action of $\otimes_{i=1}^mA_i$ on $Tor^{A'}_{*,*}(k,k)$ 
is also given explicitly in Theorem \ref{thm:basis}. A minimal set of generators of $A'$ which spans $Tor^{A'}_{1,*}(k,k)$ is given in Corollary 
\ref{Cor:GPTW16}. 

Consequently the Euler-Poincar\'{e} series of $A'$ (over $k$) with variable $t$ is given by 
\[\frac{1}{P(A'; t)}=1+\sum_{I\subset\{1,\ldots,m\}\atop I\not=\emptyset}
		\sum_{n=1}^{\infty}(-1)^n
		\mathrm{dim}\wt{H}_{n-1}(K_{I};k)\prod_{i\in I}P(J(A_i);t).
\] Together with the formula $P(\underline{A}^K;t)=P(A'; t)\prod_{i=1}^mP(A_i;t)$, we see that
$P(\underline{A}^K;t)$ is a rational function of $t$, if each $P(A_i;t)$ is (see \eqref{iso:TorAK} for an explicit formula for $Tor^{\underline{A}^K}_{*,*}(k,k)$).

The structure of this paper is as follows. In Section \ref{sec:owa} we show that the classifying space 
$\overline{W}(\underline{G}^K)$ can be 
reduced to a smaller subcomplex $\underline{\overline{W}(G)}^K$, using Whitehead's Acyclic Theorem. Based on this we show in Section \ref{sec:dec} that the polyhedral product $\underline{G}^K$ gives a model of the loop space $\Omega\underline{X}^K$, with $G_i=G(X_i)$ the Kan construction on $X_i$. The isomorphism $H_*(k\underline{G}^K)\cong \underline{A}^K$ is given in Section \ref{sec:oth}. In Section \ref{sec:loop} we prove a version of Whitehead's Acyclic Theorem in $\mathcal{H}_k^c$ and calculate $Tor^{A'}_{s,*}(k,k)$ using the BBCG spectral sequence in \cite{BBCG17}. Section 
\ref{sec:last} is an appendix in which we prove some well-known facts on the structure theorems of graph products 
of groups and connected Hopf algebras, respectively, which we believe useful for the readers. 

\section{On Whitehead's Acyclic Theorem}\label{sec:owa}
Let $K$ be an abstract simplicial complex with vertices labeled by $1,\ldots,m$, namely $K$ is a collection of subsets of $V=\{1,\ldots,m\}$ (such a subset is called a simplex in $K$), including the empty set, such that whenever $\sigma\in K$ we have $\tau \in K$ if $\tau\subset\sigma$. 
Let $\mathcal{K}$ be the category whose objects are simplices of $K$ and morphisms are inclusion of sets.

Following \cite{PR08,PRV04} , let $\mathcal{M}$ be a monoidal category with a functor 
$\otimes\co M\times M\to M$ and an object $1$ such that natural isomorphisms
\[
          M\otimes(M'\otimes M'')=(M\otimes M')\otimes M'', \quad 1\otimes M=M=M\otimes 1
\]   
hold for objects $M,M',M''$ of $\mathcal{M}$. 
Let $(\underline{M},\underline{M}')=(M_i,M'_i)_{i=1}^m$ be $m$ pairs of objects in $\mathcal{M}$, 
together with morphisms $\iota_i\co M_i'\to M_i$, $i=1,\ldots,m$. 
Consider a functor $Z\co \mathcal{K}\to \mathcal{M}$ such that $Z(\sigma)=\otimes_{i=1}^m Y_i$, where $Y_i=M_i$ if $i\in\sigma$ and 
otherwise $Y_i=M'_i$; a morphism $Z(\tau\to\sigma)$ is induced by that of each component which is either
$\mathrm{id}_{M_i}$ or $\iota_i$. 
The colimit of $Z$, if it exists, is called a \emph{polyhedral product} in $\mathcal{M}$. 
Clearly there is a natural morphism 
\begin{equation}
Z\to\otimes_{i=1}^mM_i. \label{def:Ab}
\end{equation}
 In particular, we shall use the notation $\underline{M}^K$ to denote the colimit, if $1$ is initial in $\mathcal{M}$ and 
 $M_i'=1$ for all $i=1,\ldots,m$.

We say that the simplicial complex $K$ is \emph{flag}, if it is the maximal one with the same $1$-skeleton. More precisely, any subset $\sigma\subset\{1,\ldots,m\}$ such that $\{i,j\}\in K$ for every pair $i,j\in\sigma$ implies 
$\sigma\in K$. Unless otherwise stated, $K$ will be a flag complex in what follows. A \emph{full} \emph{subcomplex}
 $K_I$ of $K$ with respect to $I\subset\{1,\ldots,m\} $ is given by the collection $\{\sigma\in K\mid\sigma\subset I\}$. Clearly $K_I$ is also flag.

Let $\underline{G}^K$ be the polyhedral product in the category $\mathcal{G}$ 
of discrete groups, endowed with direct product and $1=1_{\mathcal{G}}$ the trivial group. 
It is the colimit of $\mathcal{K}\to \mathcal{G}$, where $\underline{G}=(G_i,1)_{i=1}^m$. 
Since $K$ is determined by its $1$-skeleton $K^1$, it turns out that $G$ is the graph product of $G_i$ along
the simplicial graph $K^1$, namely the free product $\coprod_{i=1}^mG_i$ subject to the 
relations $[G_i,G_j]=1$ whenever 
$\{i,j\}\in K^1$ (see Proposition \ref{prop:preG}).   

Let $\mathcal{S}_*$ be the category of pointed simplicial sets, endowed with cartesian products and a given point. 
Recall that a \emph{simplicial group} is a simplicial object in $\mathcal{G}$. For a simplicial group $G$, the 
simplicial classifying space $\overline{W}(G)$ of Eilenberg-MacLane is a simplicial set whose $n$-simplices are in the 
form $(g^{n-1},g^{n-2},\ldots, g^0)$ with $g^i\in G$ a simplex of dimension $i$. Moreover, for the face and degeneracy maps, 
$d_j(g^{n-i})_{i=1}^n$ is given by
\[
       \begin{cases} (g^{n-i})_{i=2}^n &   j=0\\
       d_{j-1}g^{n-1},d_{j-2}g^{n-2},\ldots,   (d_0g^{n-j})\cdot g^{n-j-1}, g^{n-j-2},\ldots, g^0) & 0<j<n\\
       (d_{n-2}g^{n-1}, \ldots, d_{n-i-1}g^{n-i}, \ldots, d_1g^2, d_0g^1) & j=n,
       \end{cases}
\]
and $s_j((g^{n-i})_{i=1}^n)=(s_{j-1}g^{n-1},\ldots, s_0g^{n-j},1,g^{n-j-1},\ldots, g^0)$, $j=0,\ldots,n$.

Notice that when $L\subset K$ is a full subcomplex, we have a canonical inclusion
$\underline{G}^L\to \underline{G}^K$ of polyhedral products in $\mathcal{G}$, whence the canonical inclusion
$\overline{W}(\underline{G}^L)\to \overline{W}(\underline{G}^K)$ of classifying spaces. 
Likewise, we have a canonical inclusion $\underline{\overline{W}(G)}^L\to\underline{\overline{W}(G)}^K$
of polyhedral products in $\mathcal{S}_*$, with $\underline{\overline{W}(G)}=(\overline{W}(G_i),*)_{i=1}^m$. 

Here is our main theorem in this section, where we write an $n$-simplex of $\underline{\overline{W}G}^K$
 as $(x_i)_{i=1}^m$, $x_i\in\overline{W}(G_i)$, i.e., 
 being identified with its image in $\prod_{i=1}^m\overline{W}(G_i)$.  
	\begin{thm}\label{thm:weakequi}
	Let $K$ be a flag complex and let $G_i$ be a simplicial group, $i=1,\ldots,m$. 
	Then the simplicial inclusion
	\begin{equation}
	      \phi\co \underline{\overline{W}(G)}^{K}\stackrel{\simeq}{\longrightarrow}
	       \overline{W}(\underline{G}^{K}). \label{def:phi}
	\end{equation}
	that sends an $n$-simplex $(x_k)_{k=1}^m$, 
	$x_k=(g^{n-i}_k)_{i=1}^{n}\in \overline{W}(G_k)$, to $(g^{n-i})_{i=1}^{n}$ with 
	$g^{n-i}=g^{n-i}_1g^{n-i}_2\cdots g^{n-i}_m\in \underline{G}^K$, is 
	a weak homotopy equivalence.
	Moreover, $\phi$ is natural in the following sense: let $(K,L)$ be a pair 
	of flag complexes with 
	$L\subset K$ a full subcomplex, and let $f_i\co G_i\to\wt{G}_i$ be a morphism of simplicial 
	groups, $i=1,\ldots,m$,
	which induce a morphism $\underline{f}^K\co\underline{G}^K\to
	\underline{\wt{G}}^{K}$ (resp. $\underline{f}^L\co\underline{G}^L\to
	\underline{\wt{G}}^{L}$) of
	polyhedral products of groups,
        then the diagram
	\begin{equation}
	\xymatrix{
	& \underline{\overline{W}\wt{G}}^L \ar[rr]^{\wt{\phi}}\ar[dd]& 
	&\overline{W}(\underline{\wt{G}}^L)\ar[dd]\\
	\underline{\overline{W}{G}}^L \ar[rr]^{\phi}\ar[dd]\ar[ru]^{\underline{\overline{W}f}^L}&
	&\overline{W}(\underline{G}^L)\ar[dd]\ar[ru]^{\overline{W}\underline{f}^L}&\\
	& \underline{\overline{W}\wt{G}}^K \ar[rr]^{\wt{\phi}} & &\overline{W}(\underline{\wt{G}}^K)\\
	\underline{\overline{W}{G}}^K \ar[rr]^{\phi}\ar[ru]^{\underline{\overline{W}f}^K}&
	&\overline{W}(\underline{G}^K)\ar[ru]^{\overline{W}\underline{f}^K}&
	}\label{cd:GLK}
        \end{equation}
        commutes, in which all vertical morphisms are canonical inclusions.
	\end{thm}
	The proof is based on an induction using the 
	following form of Whitehead's Acyclic Theorem, together with a lemma below.
	\begin{thm}[Kan--Thurston \cite{KT76}]\label{thm:Whitehead}
	Let $G_1, G_2$ and $G_3$ be three simplicial groups together with two simplicial monomorphisms 
	$G_3\to G_i$, $i=1,2$. Then the inclusions 
	$\overline{W}(G_i)\to\overline{W}(G_1\ast_{G_3}G_2)$, $i=1,\ldots,3$,
	induce a weak homotopy equivalence
	\[i_{\overline{W}}\co \overline{W}G_1\coprod_{\overline{W}G_3}\overline{W}G_2
	\stackrel{\simeq}{\longrightarrow} \overline{W}(G_1\ast_{G_3} G_2).\]
	Here $G_1\ast_{G_3} G_2$ is the amalgam of $G_1, G_2$ along $G_3$.
	\end{thm}
         Since $i_{\overline{W}}$ is defined by canonical inclusions 
         $\overline{W}(G_i)\to\overline{W}(G_1\ast_{G_3} G_2)$, it is also natural. More explicitly, 
	if we have morphisms $\alpha_i\co G_i\to\wt{G}_i $ of simplicial groups 
	$i=1,2,3$, together with
	simplicial monomorphisms $\wt{G}_3\to \wt{G}_i$, $i=1,2$, then we have a 
	commutative diagram
	\[
	\begin{CD}
	\overline{W}G_1\coprod_{\overline{W}G_3}\overline{W}G_2
	@>i_{\overline{W}}>> \overline{W}(G_1\ast_{G_3} G_2)\\
	@VVV                        @VVV \\
	\overline{W}\wt{G}_1\coprod_{\overline{W}\wt{G}_3}\overline{W}\wt{G}_2
	@>i_{\overline{W}}>> \overline{W}(\wt{G}_1\ast_{\wt{G}_3} \wt{G}_2),
	\end{CD}
	\]
        and Theorem \ref{thm:Whitehead} implies that the horizontal morphisms are weak homotopy equivalences.
\begin{lem}\label{lem:L11}
	Let $K$ be a flag complex with at least two vertices. 	
	Then either $K$ is a simplex, or $K$ admits a splitting of two non-empty 
	full subcomplexes $K',K''$, so that their intersection $K'\cap K''$ (may be empty) 
	is a full subcomplex of $K$. 
\end{lem}
\begin{proof}
	Let $v$ be a vertex of $K$. We choose 
	$K'_v=\mathrm{Star}_v$ the star of $v$ in $K$, and let $K''_v=\mathrm{Star}_v^C$ be 
	its complement, 
	as the union of those simplices	which does not contain $v$ as a vertex. 
	If $K''_v$ is empty, then $K$ is 
	a cone with apex $v$, so we choose another vertex $v'$ and repeat the procedure above.

	If $K''_v$ is empty for all vertices $v$, then any pair of two vertices of $K$ is connected by an 
	edge, whence $K$ is a simplex as being flag. Otherwise there exists a non-trivial pair 
	$K'_v$ and $K''_v$. Clearly they are full subcomplexes. 
	Now suppose $L=K'\cap K''$ is non-empty (otherwise we are done). 
	If $L$ is not full, there exists a non-empty simplex of $K$, 
	whose vertices belong to $L$, but itself does not. Let $\sigma$ be a minimal one with this 
	property, namely all proper faces of $\sigma$ belong to $L$. We see that the simplex 
	spanned by $v$ and $L$ does not belong to $K$, 
	but all its proper faces do, which is a contradiction
	since $K$ is flag.
\end{proof}

\begin{proof}[Proof of Theorem \ref{thm:weakequi}]
        It is straightforward to check that
        Diagram \eqref{cd:GLK} is commutative, from the definitions. It remains to prove that 
        $\phi$ is simplicial and induces a weak homotopy equivalence.
        
        To see that $\phi$ is simplicial, one checks directly that 
	on each $n$-simplex, $\phi$ commutes with all degeneracy maps 
$s_j$, $j=0,\ldots,n$, as well as face maps $d_0$ and $d_n$. For $d_j$ with $0<j<n$ it suffices
to show
\[\prod_{i=1}^n(d_0g^{n-j}_i)g^{n-j-1}_i=\prod_{i=1}^nd_0g^{n-j}_i\prod_{i=1}^ng^{n-j-1}_i,\]
which follows from the definition of the polyhedral products, that a simplex 
$(g_{i}^*)_{i=1}^m\in \underline{\overline{W}(G)}^K$ implies $\{i\mid g_i^*\not=1\}\in K$,
hence $g_i^*g_{i'}^*=g_{i'}^*g_{i}^*$, for all $i,i'\in\{1,\ldots,m\}$.   
          
To show that $\phi$ induces a weak homotopy equivalence, 
first consider the case when $K=\Delta$ is a simplex 
         with $m$ vertices. Now $\underline{G}^K=\prod_{i=1}^mG_i$, the 
         direct product, and $\underline{\overline{W}(G)}^K=\prod_{i=1}^m\overline{W}(G_i)$.
          By definition we have an isomorphism
        $\phi\co\prod_{i=1}^m\overline{W}(G_i)\to \overline{W}(\prod_{i=1}^mG_i)$ of simplicial sets. 
		
	In general we use an induction on the number $m$ of vertices of $K$, which is trivial when $m=1$. 
	Now suppose $\phi$ is a weak homotopy equivalence for every simplicial complex whose number of vertices 
	is smaller than $n\geq 2$, and let $K$ be a flag complex with $n$ vertices.
	By Lemma \ref{lem:L11}, either $K$ is a simplex which is already proved as above, 
	or $K=K_1\coprod_{K_3} K_2$ is the union of two
	proper flag complexes $K_1, K_2$ along  their intersection $K_3$, which is clearly also flag. 
	Consider the diagram
	\[\xymatrix{
	                                \underline{\overline{W}(G)}^{K_1}\coprod_{\underline{\overline{W}(G)}^{K_3}}\underline{\overline{W}(G)}^{K_2}   \ar[rd]  \ar[d]_{\phi_1\coprod_{\phi_3}\phi_2} & \\
	                                  \overline{W}(\underline{G}^{K_1})
	    \coprod_{\overline{W}(\underline{G}^{K_3})}\overline{W}(\underline{G}^{K_2}) \ar[r]^{i_{\overline{W}}}   &
	    \overline{W}(\underline{G}^{K_1}*_{\underline{G}^{K_3}}\underline{G}^{K_2})
	}
	   	\]
		of simplicial inclusions, in which $\phi_i$ is a weak homotopy equivalence 
		by induction hypothesis, $i=1,2,3$, we see that $\phi_1\coprod_{\phi_3}\phi_2$ is a weak homotopy equivalence by the Gluing Lemma 
		(see for example, \cite[Lemma 2.4, p. 124]{WZZ99}). Together with Theorem \ref{thm:Whitehead}, we finish the induction, since 
		$ \underline{G}^{K_1}*_{\underline{G}^{K_3}}\underline{G}^{K_2}=\underline{G}^K$ by the uniqueness of
		colimits in $\mathcal{G}$ (see Lemma \ref{lem:decG} for details).  
	\end{proof}

\begin{cor}[Stafa \cite{Sta15}]\label{Cor:Sta15}
Suppose $K$ is a flag complex and $X_i$ is a connected $CW$ complex having a 
single homotopy group $G_i$ in dimension $1$, 
namely $X=K(G_i,1)$, $i=1,\ldots,m$. Then up to homotopy,
$\underline{X}^{K}=\overline{W}(\underline{G}^K)=K(\underline{G}^K,1)$. 
\end{cor}
\begin{proof}
We consider $G_i$ as a simplicial group whose non-degenerate simplices concentrate in dimension $0$, being 
$G_i$ itself. It is easily checked that now the classifying space $\overline{W}-$ coincides with the usual one for groups.
 Up to homotopy, we replace $X_i$ by $\overline{W}(G_i)$. Then
$\underline{X}^{K}= \overline{W}(\underline{G}^K)$ up to homotopy, by Theorem \ref{thm:weakequi}. 
\end{proof}

\section{Decomposition of a graph product of simplicial groups}\label{sec:dec}
	Let
	$\underline{G}^K$ be the polyhedral product of groups $\underline{G}=(G_i)_{i=1}^m$ with
	respect to the flag complex $K$. 
	
	\begin{defin}\label{def:order}
	We say that a word of the form
	\begin{equation}
	    g=\prod_{k=1}^ng_{i_k,k}=g_{i_1,1}g_{i_2,2}\ldots g_{i_n,n},  \label{word:g}
	\end{equation}
	with $1\not=g_{i_k,*}\in G_{i_k}$ (the identity $1$ is the empty word), 
	is \emph{reduced} in $\underline{G}^K$,
if the two operations of I) 
exchanging $g_{i_k,k}g_{i_{k+1},k+1}$ into $g_{i_{k+1},k+1}g_{i_k,k}$, when 
$\{i_k,i_{k+1}\}\in K$ and II) merging $g_{i_{k},k}g_{i_{k+1},k+1}$ into a single element
$g''_{i_{k},k}$ (we delete it when we get the identity of $G_{i_k}$), 
when $i_k=i_{k+1}$, will not make $g$ a word of shorter length. 
The \emph{lexicographic partial order} on reduced words 
	is given by rules 
	1) $1<g$ if $g\not=1$, $g_{1}<g_{2}<\ldots<g_{m}$ for every $1\not=g_i\in G_i$, 
	and 2) two reduced words
	$g=\prod_{k=1}^ng_{i_k,k}<g'=\prod_{k=1}^{n'}g'_{i_{k}',k}$, 
	if there exists an index $j\leq n$, 
	such that $g_{i_k,k}=g'_{i_{k}',k}\in G_{i_k}$ 
	for $k\leq j$ while $i_{j+1}<i_{j+1}'$ (if $j=n$, this means
	 $n'>n$ and the first $n$ letters of $g'$ coincide with that of $g$).
	
	A reduced word \eqref{word:g} 
is \emph{locally minimal} if a single
operation I), whenever possible,  
will make it  larger in the partial order above. 
\end{defin}
Notice that a word of the form \eqref{word:g} is reduced implies that $i_k\not=i_{k+1}$, $k=1,\ldots, n-1$.
We need the following well-known structure theorem of a graph product of groups (a proof is given
in Section \ref{sec:last}).
    	\begin{thm}\label{thm:normal}
	Every element in the polyhedral product $\underline{G}^K$ has a locally minimal 
	presentation 
	$g=\prod_{k=1}^ng_{i_k,k}$ with $1\not= g_{i_k,k}\in G_{i_k}$.
	Moreover, the presentation is unique in the sense that 
	for two locally minimal presentations 
	$g=\prod_{k=1}^{n}g_{i_k,k}$ and $g'=\prod_{k=1}^{n'}g'_{i_k',k}$,  if $g=g'$, then 
	$n=n'$ and $g_{i_k,k}=g'_{i_k',k}\in G_{i_k}$, $k=1,\ldots,n$. 
	The presentation of the multiplication 
	\[gg'=\prod_{k=1}^{n}g_{i_k,k}\prod_{k=1}^{n'}g'_{i_k',k}
	\]
	becomes locally minimal after a finite sequence of operations that 
	I) exchange two adjacent letters, if possible, and II) merge them
	into a single one with the multiplication in $G_i$ 
	if they both come from $G_i$, $i=1,\ldots,m$.  
	\end{thm}

For $i=1,\ldots,m$, let 
\begin{equation}
\pi_i\co \underline{G}^K\to G_i \label{def:pii}
\end{equation}
be the projection sending a locally minimal
presentation $g=\prod_{k=1}^ng_{i_k,k}$ to $\pi_i(g)=\prod_{i_k=i}g_{i_k,k}$, the product of 
components in $G_i$. 
\begin{lem}\label{lem:monoid}
	For $i=1,\ldots,m$, 
	the projection $\pi_i\co \underline{G}^K\to G_i$ is a morphism of monoids.
\end{lem}
\begin{proof}
Let $g$ and $g'$ be two elements of $\underline{G}^K$ in their locally minimal presentations 
$g=\prod_{k=1}^ng_{i_k,k}$ and $g'=\prod_{k=1}^{n'}g'_{i'_k,k}$, respectively. Consider the word
\begin{equation}
	gg'=g_{i_1,k}\cdots g_{i_n,n}g'_{i'_1,1}\cdots g'_{i'_{n'},n'}. \label{prod:gg'}
\end{equation}
Notice that $\pi_i$ preserves the two operations in Definition \ref{def:order}, 
hence we can apply $\pi_i$ 
directly on the word \eqref{prod:gg'}, whence $\pi_i(gg')=\pi_i({g})\pi_i(g')$.
\end{proof}

Let $\mathrm{Ab}\co \underline{G}^K \to \prod_{i=1}^m G_i$ be the canonical epimorphism 
sending $g$ to $(\pi_i(g))_{i=1}^m$.  
It can be checked that this morphism coincides with the morphism of polyhedral products 
which is obtained from inclusion $K\to\Delta$ to
the simplex with the same vertices.
	We have an exact sequence 
	\[
	\begin{CD}
	1 @>>> H @>>> \underline{G}^K @>\mathrm{Ab}>> \prod_{i=1}^m G_i @>>>1
	\end{CD}
	\]
	where $H=\ker\mathrm{Ab}$. 
	\begin{thm}\label{thm:JW}
	Let $\underline{G}^K$ be the polyhedral product with respect to 
	simplicial groups $\underline{G}=(G_i,*)_{i=1}^m$ and a flag complex $K$. 
	There is an isomorphism 
	\[ \begin{CD}
	      \underline{G}^K @>\phi>\cong > H \times \prod_{i=1}^m G_i
	  \end{CD}
	\]
	of simplicial sets preserving the left multiplication of $H$, 
	which is natural in the sense that if we have a simplicial inclusion 
	$K\to \wt{K}$ of flag complexes with the same vertices, together with 
	morphisms $f_i\co G_i\to 
	\wt{G}_i$ 
	of simplicial groups, 
	$i=1,\ldots,m$, then the diagram
	\begin{equation} \begin{CD}
	      \underline{G}^K @>\psi>\cong > H \times \prod_{i=1}^m G_i\\
	      @V\underline{f}VV                               @V(\underline{f}', \prod_i f_i)VV  \\
	      \underline{\wt{G}}^{\wt{K}}@>\wt{\psi}>\cong > \wt{H} \times \prod_{i=1}^m \wt{G}_i\\
	  \end{CD}\label{diag:phi}
	\end{equation}
	commutes, where $\underline{f}'\co H\to \wt{H}$ is the morphism induced by $\underline{f}$.
	\end{thm}
	\begin{proof}
	(Following Jie Wu \cite{Wu10}.)
	Given $g\in \underline{G}^K $, let $\gamma_i=\pi_i(g)\in G_i$ as in \eqref{def:pii}, and let
	\[\pi_H(g)=h=g\prod_{i=1}^m\gamma_i^{-1} = g\gamma_1^{-1}\cdots\gamma_m^{-1}.\]
	Clearly $h\in\ker(\mathrm{Ab})=H$, whence $\pi_H\co\underline{G}^K\to H$ is well-defined..
	Let $\psi=(\pi_H,\prod_{i=1}^n\pi_i)$ and it suffices to show that it is 
	a simplicial bijection. It is straightforward to 
	check that $\psi$ commutes with all face and degeneracy maps, which is also injective since
	$g=\pi_H(g)\prod_{i=0}^{m-1}\pi_{m-i}(g)$. To prove the surjectiveness, let $h\in H$ be given, as well as 
	$\gamma_i\in G_i$, $i=1,\ldots,m$. After writing $h$ as a product of commutators 
	(see Theorem \ref{thm:PV} below) we see that 
	$\pi_H(h)=h$, since $\pi_i(h)=1$, $i=1,\ldots,m$. By Lemma \ref{lem:monoid} we have 
	\[ \psi(g)=(h,(\gamma_i)_{i=1}^m),\]
	where $g=h\gamma_m\cdots\gamma_1$. Clearly 
	$\pi_H$ preserves the left multiplication of $H$. 
	
       Finally, Diagram \eqref{diag:phi} is commutative since all constructions above 
       are natural.
	\end{proof}
Recall that for a path-connected and pointed space $X$, the Kan construction $GSing(X)$ of the simplicial set $Sing(X)$ (with a single simplex in dimension $0$, up to homotopy) gives a model for the Moore loop space $\Omega X$. 
Moreover, there are canonical maps $f\co |Sing(X)|\to X$ of topological spaces and 
$\alpha\co |GSing(X)|\to \Omega X$
of topological monoids (sending a simplicial circuit into a continuous Moore loop) respectively, each inducing a weak homotopy equivalence (see \cite{Kan58}). Moreover,
there is a canonical morphism $\beta\co Sing(X)\to \overline{W}(GSing(X))$ of simplicial sets inducing a homotopy equivalence.

Now let $\iota\co \underline{X}^{K}\to \prod_{i=1}^mX_i$ be the inclusion of polyhedral 
	products with respect to a flag complex $K$ and $\underline{X}=(X_i,*)_{i=1}^m$ with each $X_i$ a 
	path-connected and pointed $CW$ complex, and let $F$ be the homotopy fiber of $\iota$. 
	Let $G_i=GSing(X_i)$ be the Kan construction, $i=1,\ldots,m$; the homotopy equivalences 
	$f_i\co |Sing(X_i)|\to X_i$ and $\beta_i\co Sing(X)\to \overline{W}(GSing(X))$ give rise to
	morphisms 
	        \begin{equation}
        \begin{CD}
        |\underline{\overline{W}(G)}^K| @<\underline{\beta}^K<<\underline{|Sing(X)|}^K@>\underline{f}^K>>\underline{X}^K
        \end{CD}\label{diag:HC}
        \end{equation}
        of polyhedral products in $Top_*$, which are homotopy equivalences. This follows 
        from a general fact that the identification $*\to Z_i$ is a cofibration 
        with $Z_i=X_i,|Sing(X_i)|$ or $|\overline{W}(G_i)|$, hence the colimit is homotopy equivalent to the homotopy
        colimit, and then by the comparison theorem (see \cite{BBCG10,WZZ99}). 
        Now we replace $\underline{X}^K$ by  $|\underline{\overline{W}(G)}^K|$, and 
        consider the commutative diagram
	\[
	\begin{CD}
	 G\underline{\overline{W}(G)}^K@>>> G\underline{\overline{W}(G)}^K\times_{\tau} \underline{\overline{W}(G)}^K@>>>\underline{\overline{W}(G)}^K\\
	    @VG\phi VV                                        @V(G\phi, \phi)VV                        @V\phi VV  \\
	G\overline{W}(\underline{G}^K) @>>>G\overline{W}(\underline{G}^K)\times_{\tau}\overline{W}(\underline{G}^K)
	 @>>>\overline{W}(\underline{G}^K)\\
	@VAd_\mathrm{id}VV           @V (Ad_\mathrm{id},\mathrm{id})VV                       @V\mathrm{id}VV\\
	      \underline{G}^K @>>> \underline{G}^K\times_{\tau}\overline{W}(\underline{G}^K)@>>>\overline{W}(\underline{G}^K)
	\end{CD}
	\]
	in which the rows are simplicial principal bundles, $\phi$ is the inclusion in Theorem \ref{thm:weakequi} and $Ad_{\mathrm{id}}\co G\overline{W}(\underline{G}^K)\to (\underline{G}^K)$ is the adjunction of the identity of $\overline{W}(\underline{G}^K)$.
	We have three contractible total spaces in the middle column, 
	each with a twisting function $\tau$ (see \cite{Kan58b}). A comparison of homotopy groups in the long
	exact sequences shows that both morphisms $G\phi$ and $Ad_{\mathrm{id}}$ of simplicial groups induce 
	homotopy equivalences, so does their composition 
	\begin{equation}
	\eta\co G\underline{\overline{W}(G)}^K\to \underline{G}^K.\label{def:eta}
	\end{equation}
        Notice that the naturality of the functor $G$ and $\overline{W}$ gives a commutative diagram
        \begin{equation}\label{diag:HF}
	\begin{CD}
	\Omega F @>>> \Omega |\underline{\overline{W}(G)}^K|@>\Omega \iota >>  \prod_{i=1}^m
	\Omega|\overline{W}(G_i)|\\
	@A A\alpha' A                     @A A \alpha A       @A A\alpha''=\prod_{i=1}^m\alpha_i A\\
	|G_F| @>>>      |G\underline{\overline{W}(G)}^K | @>|G\iota|>> |\prod_{i=1}^mG\overline{W}(G_i)| \\
	 @VV\eta' V                           @VV\eta V                          @VV\eta''=\prod_{i=1}^mAd_iV\\
	        |H|     @>>>        |\underline{G}^K|@>\mathrm{Ab}>> |\prod_{i=1}^mG_i|
	\end{CD}
	\end{equation}
	of topological monoids, where $G_F=\mathrm{ker}(G\iota)$ and $H=\mathrm{ker}(\mathrm{Ab})$, respectively,
	such that all vertical maps are morphisms of topological monoids each inducing a homotopy equivalence (in which $\alpha\co |G-|\to\Omega|-|$ sends simplicial circuits to Moore loops).
       Together with the splitting $\underline{G}^K=H\times\prod_{i=1}^mG_i$ which preserves the left multiplication of $H$, 
       we have the conclusion below, which is a special case of \cite[Lemma 3.16, p. 15]{The18}.
       	\begin{thm}\label{thm:dec_loops}
		Let $\iota\co \underline{X}^{K}\to \prod_{i=1}^mX_i$ 
		(resp. $\wt{\iota}\co \underline{\wt{X}}^{\wt{K}}\to \prod_{i=1}^m\wt{X}_i$) 
		be the inclusion of polyhedral 
		products with respect to a flag complex $K$ (resp. $\wt{K}$) 
		and $\underline{X}=(X_i,*)_{i=1}^m$ (resp. $\underline{X}=(\wt{X}_i,*)_{i=1}^m$)
		with each $X_i$ (resp. $\wt{X}_i$) a 
		path-connected and pointed $CW$ complex. Let $F$ (resp. $\wt{F}$) 
		be the homotopy fiber of $\iota$ (resp. $\wt{\iota}$).
		Then up to homotopy, the corresponding fibration of Moore loop spaces admits a splitting
	\[
	           \Omega \underline{X}^{K}\simeq \Omega F\times \Omega \prod_{i=1}^mX_i,
	\]
	preserving the left actions of the monoid $\Omega F$. Moreover, the diagram
	\[\begin{CD}\Omega \underline{X}^{K} @>>>\Omega F\times \Omega \prod_{i=1}^mX_i\\
			@VVV  @VVV\\
			\Omega \underline{\wt{X}}^{\wt{K}} @>>>\Omega \wt{F}\times \Omega 
			\prod_{i=1}^m\wt{X}_i
		\end{CD}	\]
		commutes, provided that $\underline{X}^{K}\to\underline{\wt{X}}^{\wt{K}}$ is 
		induced by a simplicial inclusion $K\to\wt{K}$ and $CW$ maps $X_i\to\wt{X}_i$,
		$i=1,\ldots,m$.
	\end{thm}
	Based on this decomposition, we give a topological interpretation of the theorem below.
Recall that the $1$-skeleton $K^1$ of a flag complex $K$ is \emph{chordal}, if every simplicial circuit in $K$ 
bounds a union of triangles in $K$.  
\begin{thm}\cite[Theorem 5.2]{PV19}\label{thm:PV}
Let $\underline{G}^K$ be the polyhedral product with respect to 
	simplicial groups $\underline{G}=(G_i,*)_{i=1}^m$ and a flag complex $K$. Consider all full subcomplexes
	$K_I$ with more than one connected component, where $I\subset\{1,\ldots,m\}$. We write all elements of 
	$I=\{i_l\}_{l=1}^{n_I}$ in an increasing order, with $n_I$ the cardinality of $I$. Let 
	$L_g(h)=(g,h)=g^{-1}h^{-1}gh$ be the commutator.  
	Then $H=\mathrm{ker}(\mathrm{Ab})$ is 
	generated by the set
	\[S=\bigcup_{\mathrm{rank}\wt{H}_0(K_I)>0} S_I,\]
	in which each $S_I$ is the collection of all iterated commutators of the form
	\begin{equation}
	   L_{g_{i_1}}\circ L_{g_{i_2}}\circ\ldots\circ L_{g_{i_{t-1}}}\circ L_{g_{i_{t+1}}}\circ\ldots
	   \circ L_{g_{i_{n}}}(g_{i_{t}}),\label{def:SI}
	\end{equation}
	such that $\{i_t\}$ is the smallest vertex in a connected component of $K_I$ not containing 
	the vertex $\{i_n\}$, with $g_{i_l}$ running through all elements of $G_{i_l}$ which is not the identity. Moreover, $S$ generates $H$ freely if and only if $K^1$ is chordal.
	\end{thm}
	
\begin{cor}\label{cor:PV}
Let $F$ be the homotopy fiber of the inclusion $\iota\co \underline{X}^{K}\to \prod_{i=1}^mX_i$ of polyhedral 
	products with respect to a flag complex $K$ and $\underline{X}=(X_i,*)_{i=1}^m$ with each $X_i$ a 
	path-connected and pointed $CW$ complex.  
There exists a morphism of topological monoids
\begin{equation}
	\Omega\Sigma\left(\bigvee_{I\subset\{1,\ldots,m\},\  \wt{b}_I>0}
\wt{b}_I\widehat{\underline{\Omega X}}^{I}\right)\to \Omega F\label{mor:TM}
\end{equation}
where $\wt{b}_I=\mathrm{rank}\wt{H}_0(K_I)$ is the reduced Betti number in dimension $0$
 and $\wt{b}_I\widehat{\underline{\Omega X}}^I$ is the wedge sum of $\wt{b}_I$ copies of $\widehat{\underline{\Omega X}}$ with $\widehat{\underline{\Omega X}}^I=\bigwedge_{i\in I}\Omega X_i$ a smash product of Moore loop spaces.
 When $K^1$ is a chordal graph, the morphism above induces a homotopy equivalence. 
 Moreover, the morphism above induces a continuous map  
 \begin{equation}
	 \Sigma\left(\bigvee_{I\subset\{1,\ldots,m\}, \ \wt{b}_I>0}
	 \wt{b}_I\widehat{\underline{\Omega X}}^{I}\right)\to F \label{eq:F}
 \end{equation}
 which is a homotopy equivalence when $K^1$ is a chordal graph. 
\end{cor}
\begin{proof}
Again up to homotopy, we replace $\Omega X_i$ by the monoid $|G_i|$, with $G_i=GSing(X_i)$. Let 
$F\langle S_I\rangle$ be the free simplicial group generated by $S_I$, whose face and degeneracy operators follows
each of $G_i$. Notice that for each connected component of $K_I$ not containing the maximal vertex from $I$,
$f_I\co \underline{\widehat{G}}^I=\bigwedge_{i\in I}G_i\to F\langle S_I\rangle$ sending 
$(g_i)_{i\in I}$ to the generator \eqref{def:SI} is an injection of simpicial sets, and it becomes
a bijection of free simplicial groups, by the definition of $S_I$, after we extend $f_I$ to 
\begin{equation}
\wt{f}_I\co F[\bigvee\wt{b}_I\underline{\widehat{G}}^I]\to F\langle S_I\rangle \label{iso:MF}
\end{equation}
over a wedge sum of $\wt{b}_I$ copies of $\underline{\widehat{G}}^I$, where 
$F[\bigvee\wt{b}_I\underline{\widehat{G}}^I]$ denotes the free construction of Milnor \cite{Mil72}. Finally
the statement on \eqref{mor:TM} follows from Theorem \ref{thm:dec_loops}, Theorem \ref{thm:PV}, 
as $I$ runs through all subsets of $\{1,\ldots,m\}$ such that 
$\wt{b}_I>0$, and the homotopy equivalence
$|F[X]|\cong \Omega\Sigma |X|$ for any pointed simplicial set $X$. The last statement is obtained by applying 
the functor $\overline{W}$ on both sides of \eqref{mor:TM}.
\end{proof}

\begin{exm}\label{exm:Medrano}
	Consider the compact Lie group $G_k=S^k$, $k=0,1,3$ (as usual, $S^0$ is the cyclic group of order $2$ and $S^3$ is the collection of quaternions of norm $1$, which is the boundary of the ball $D^4$ as the quaternions of norm $\leq 1$).
	Let $X_i=BG_k$, namely $X_i=\mathbb{R}P^{\infty}, \mathbb{C}P^{\infty}$ or $\mathbb{H}P^{\infty}$, 
	so that $\Omega X_i\simeq S^k$, 
	$i=1,\ldots,m$, and let $K$ be the boundary complex of 
	a polygon with $m$ vertices. We have a homotopy equivalence 
	\[\Sigma|K_I|\simeq \begin{cases}\underbrace{S^1\vee\cdots\vee S^1}_{\widetilde{b}_I} & \wt{b}_I>0\\
	* & \wt{b}_I=0,\end{cases}\] 
	for every full subcomplex $K_I$ such that $I\not=\{1,\ldots,m\}$. 
	Notice that \eqref{eq:F} becomes
		\begin{equation}
	   \bigvee_{I\subset \{1,\ldots,m\}}\widetilde{b}_I \Sigma\widehat{S^k}^I\simeq 
	   \bigvee_{I\subsetneq\{1,\ldots,m\}}(\Sigma |K_I|\wedge \widehat{S^k}^I)\to F_k.\label{dec:F}
	\end{equation}
 Clearly $G_k$ acts on the pair $(D^{k+1},S^k)$ where the action is free on the boundary. Recall that for a given pair $(X,X')$ of pointed $CW$ complexes, we have the polyhedral product $\mathcal{Z}_K(X,X')=\cup_{\sigma\in K}B_{\sigma}$ with $B_{\sigma}=\prod_{i=1}^m Y_i$ where $Y_i=X$ if $i\in\sigma$ otherwise $Y_i=X'$ if $i\not\in\sigma$. The componentwise action of $G^m_k$ on $X^m$ as a cartesian product of $G=S^k$ acting on $X=D^{k+1}$, gives rise to an action of $G^m_k$ on $\mathcal{Z}_K(D^{k+1},S^k)$. Following \cite{DJ91}, a key observation is that in the fibration   
\begin{equation}\begin{CD}
       \mathcal{Z}_K(D^{k+1},S^k) @>>> EG^m_k\times_{G^m_k} \mathcal{Z}_K(D^{k+1},S^k) @>>> BG^m_k
\end{CD}\label{eq:DJH}
\end{equation}
obtained from the Borel construction of the action above, we have   
\[EG^m_k\times_{G^m_k} \mathcal{Z}_K(D^{k+1},S^k)=\mathcal{Z}_K(EG_k\times_{G_k} D^{k+1}, EG_k\times_{G_k} S^k) \]
by checking each component. Moreover, using an argument of homotopy colimits, the homotopy equivalence $(BG_k,*)\simeq (EG_k\times_{G_k} D^{k+1}, EG_k\times_{G_k} S^k)$ gives rise to a homotopy equivalence $\underline{BG_k}^K\simeq EG_k^m\times_{G^m_k} \mathcal{Z}_K(D^{k+1},S^k)$, whence a homotopy fibration
 \[\begin{CD}
       \mathcal{Z}_K(D^{k+1},S^k) @>>> \underline{BG_k}^K@>>> (BG_k)^m
\end{CD}
\]
obtained from \eqref{eq:DJH} (see \cite[Lemma 2.3.2]{DS07} for a different approach). The homotopy fibration implies a homotopy equivalence $F_k\simeq \mathcal{Z}_K(D^{k+1},S^k)$. Together with \cite[Theorem 2.21]{BBCG10} on the decomposition of the suspension of a polyhedral product, we have  a homotopy equivalence
	\begin{equation}
	 \Sigma F_k\simeq \bigvee_{I\subset\{1,\ldots,m\}}\Sigma^2|K_I|\wedge \widehat{S^k}^I.  \label{eq:BBCG0}
	 \end{equation}
Topologically, by \cite{BBCG15}, we have a homeomorphism $\mathcal{Z}_K(D^4,S^3)=\mathcal{Z}_{K(J')}(D^1,S^0)$ (resp. $\mathcal{Z}_K(D^2,S^1)=\mathcal{Z}_{K(J)}(D^1,S^0)$) where $K(J')$ (resp. $K(J)$) is a new simplicial complex obtained by iterated wedge constructions on $K$. Moreover, $\mathcal{Z}_K(D^4,S^3)$ and $\mathcal{Z}_K(D^2,S^1)$ have isomorphic ungraded cohomology rings (see \cite[Corollary 7.6]{BBCG15}). Using surgery theory, Gitler and L\'{o}pez de Medrano proved  that  $\mathcal{Z}_{K(J')}(D^1,S^0)$ (resp. $\mathcal{Z}_{K(J)}(D^2,S^1)$) is homeomorphic to a connected sum of sphere products whenever $\mathcal{Z}_{K}(D^1,S^0)$ is (see \cite[Theorem 2.4]{GLdM13}). As a conclusion, $\mathcal{Z}_K(D^{k+1},S^k)$ homeomorphic to a connected sum of sphere products.
 
To understand $F_k$ more explicitly we need a result of McGavran \cite{McG79} (see also \cite[Example 6.4]{BM06}), giving an explicit homeomorphism
\[
\mathcal{Z}_K(D^2,S^1)=\sharp_{j=1}^{m-3}j\binom{m-2}{j+1}S^{2+j}\times S^{m-j}
\]
where the enumeration of connected sums also holds for $k=0,3$, by \eqref{eq:BBCG0}. Consequently we have a homotopy equivalence
\begin{equation}
   F_k\simeq \sharp_{j=1}^{m-3}j\binom{m-2}{j+1}S^{k(j+1)+1}\times S^{k(m-j-1)+1}. \label{eq:Fk}
\end{equation}

Finally, after a comparison of  \eqref{eq:BBCG0},\eqref{eq:Fk} and \eqref{dec:F} we see that 
that the left-hand side of \eqref{dec:F} gives all cells of $F_k$, except one of the top dimension coming from $K$ itself (as a full subcomplex).
 Moreover, the attaching of the last cell gives the relations to obtain the correct homotopy type of $F_k$  (compare \cite{GSIP22}).  To complete this argument we need to show that \eqref{dec:F} induces an injection of homology groups, however, this requires further work.
\end{exm}
	
\section{On the homology}\label{sec:oth}
Let $\underline{G}^K$ be the polyhedral product of simplicial groups $\underline{G}=(G_i)_{i=1}^m$ with respect to
a flag complex $K$. The homotopy equivalence $\eta\co G\underline{\overline{W}(G)}^K\to \underline{G}^K$ of 
simplicial groups (see \eqref{def:eta}) enables us to understand the homology 
$H_*(\Omega |\underline{\overline{W}(G)}^K|)$ of the loop space on the polyhedral product 
$|\underline{\overline{W}(G)}^K|$.

We say that the \emph{length} of an element $g\in \underline{G}^K$ is $n$, denoted by $l(g)=n$, if $g$ has a locally minimal 
presentation $g=\prod_{k=1}^ng_{i_k,k}$ with 
$1\not=g_{i_k,k}\in G_{i_k}$ (see Theorem \ref{thm:normal}). 
Notice that $l(g)=0$ if and only if $g=1$, and $l(g)=1$ if and only if $1\not=g\in G_i$, for some $i$.
Consider the filtration on $\underline{G}^K$ such that 
\begin{equation}
F_n\underline{G}^K=\{g\in\underline{G}^K\mid l(g)\leq n\};\label{def:FnG}
\end{equation}
$F_n\underline{G}^K$ is closed under face and degeneracy maps, whence a simplicial subset,
and we have $\underline{G}^K=\bigcup_{n=0}^{\infty}F_n\underline{G}^K$. Moreover, under the multiplication we 
have 
\begin{equation}
F_n\underline{G}^K\cdot F_{n'}\underline{G}^K\subset F_{n+n'}\underline{G}^K.\label{prop:fil}
\end{equation}
	 Let $M$ be the monoid generated by a single idempotent element $t$, with a single 
	 relation $t^2=t$. 
	 Consider the right-angled monoid $M^K$ associated to $K$: 
	 $M^K$ is generated by idempotent elements 
	$t_i$, $i=1,\ldots,m$, subject to the defining relations $t_it_j=t_jt_i$ when $\{i,j\}\in K$. 
	 We say that $g\in \underline{G}^K$ is \emph{supported} by a word 
	 \begin{equation}
	 w=t_{i_1}t_{i_2}\ldots t_{i_n}\in M^K, \quad 1\leq i_1,\ldots, i_n\leq m, \label{word:w}
	 \end{equation}
	  if $g$ has a locally minimal presentation $\prod_{k=1}^ng_{i_k,k}$.
	Such a word \eqref{word:w} is \emph{locally minimal} in $M^K$, namely an exchange of 
	adjacent letters from the relations will make it larger in the lexicographic order.

Let $G_w=\prod_{k=1}^nG_{i_k}$ be the direct product associated to
a locally minimal word $w=\prod_{k=1}^nt_{i_k}\in M^K$, 
and let $\mu_w\co G_w\to \underline{G}^K$ be the morphism of 
simplicial sets sending $(g_{i_k,k})_{k=1}^n$ to $\prod_{k=1}^ng_{i_k,k}$. An element of $G_w$ is said to be 
\emph{degenerate} if at least one coordinate is $1$.
 We define $J_K(\underline{G}^K)$ as a simplicial monoid with filtration (following the construction of James \cite{Jam55}), 
 so that as a simplicial set,
\[
                      F_nJ_K(\underline{G}^K)=\coprod_{w\in\underline{M}^K\atop l(w)\leq n} G_w/\sim
\]
where $g\sim g'$ if and only if $\mu_w(g)=\mu_{w'}(g')$, with $g\in G_{w}$ and $g'\in G_{w'}$. It can be checked that 
the colimit of $\mu_w$ gives rise to a bijection 
$\mu\co J_K(\underline{G}^K)\to\underline{G}^K$ of simplicial sets, preserving 
the filtrations. 

An inverse of $\mu$ is given by the locally minimal presentations. With this inverse we endow $J_K(\underline{G}^K)$
with the structure of a simplicial monoid. More explicitly, the multiplication of $(g_{i_k,k})_{k=1}^n\in G_{w}$ and 
$(g'_{i'_k,k})_{k=1}^{n'}\in G_{w'}$ is given by the presentation of 
$\prod_{k=1}^ng_{i_k,k}\cdot \prod_{k'=1}^{n'}g'_{i'_k,k}$ (after a finite sequence of operations to make it locally minimal),
 which is an element in $G_{ww'}$, possibly 
degenerate, hence is well-defined under the equivalence relation above.
\begin{prop}
For each $n\geq 1$ we have an isomorphism
\begin{equation}
 F_{n}\underline{G}^K/F_{n-1}\underline{G}^K\to\bigvee_{w\in M^K\atop l(w)=n} \widehat{G}_{w}\label{iso:Gw}
\end{equation}
of simplicial sets, sending a element $g=\prod_{k=1}^ng_{i_k,k}$ in its locally minimal presentation 
to $(g_{i_k,k})_{k=1}^n\in \widehat{G}_{w}=\wedge_{k=1}^nG_{i_k}$ in the smash product of pointed simplicial sets 
associated to the word $w$.
\end{prop}
The proposition above follows immediately from the commutative diagram
\begin{equation}
\begin{CD}
    {G}'_w @>i_{w}>> G_w @>p_w>> \widehat{G}_w\\
    @V\mu_wVV   @V\mu_wVV               @V\widehat{\mu}_wVV \\
    F_{n-1}(\underline{G}^K) @>>> F_n(\underline{G}^K) @>>> F_{n}\underline{G}^K/F_{n-1}\underline{G}^K                
\end{CD}\label{CD:FG}
\end{equation}
where $G'_w \subset G_w$ consists of degenerate elements, 
i.e., $(g_{i_k,k})_{k=1}^n$ with at least one $g_{i_k,k}=1$.


For a simplicial group $G$, we denote by $RG$ the simplicial group ring over the ring $R$ 
where the face and degeneracy operators are $R$-linear. Recall that the homology of $G$ with coefficients in $R$ 
is the homology $H_*(RG;\partial)$, with $\partial=\sum_{i}(-1)^id_i$ the alternating sum of $i$-th face maps. In what follows, we refer the readers to Theorems \ref{thm:prepolyHopf}, \ref{thm:Hopfmain} for details on the polyhedral product of connected Hopf algebras.  
\begin{thm}\label{thm:Nat}
Let $R$ be a field, $\underline{A}=(A_i)_{i=1}^m$ be $m$ connected Hopf algebras with
$A_i=H_*(RG_i;\partial_i)$ in which $G_i$ a reduced simplicial group, $i=1,\ldots,m$.  
Then we have an isomorphism
\[\psi_*\co\underline{A}^K\stackrel{\cong}{\rightarrow} H_*(\underline{G}^K;R) \]
of connected Hopf algebras. Moreover, this isomorphism is natural with respect to the morphism 
$\underline{G}^K\to\underline{\wt{G}}^{\wt{K}}$ (resp. $\underline{A}^K\to\underline{\wt{A}}^{\wt{K}}$) of polyhedral products, which is given by a simplicial monomorphism
$K\to\wt{K}$ of flag complexes with the same vertices together with morphisms $G_i\to\wt{G}_i$ of reduced simplicial groups (resp. induced morphisms $A_i\to\wt{A}_i$ of connected Hopf algebras), $i=1,\ldots,m$: the diagram
\begin{equation}
\begin{CD}
                         \underline{A}^K @>\psi_*>\cong> H_*(\underline{G}^K;R)\\
                         @VVV                              @VVV\\
                         \underline{\wt{A}}^{\wt{K}}@>\psi_*>\cong> H_*(\underline{\wt{G}}^{\wt{K}};R)\label{CD:KK}
\end{CD}
\end{equation}
commutes.
\end{thm} 
\begin{proof}
As $R$ is a field, we choose a morphism 
\[\kappa_i\co H_*(RG_i;\partial_i)\to (RG_i;\partial_i)\] 
of differential graded $R$-modules by specifying representatives in $(RG_i;\partial_i)$, with trivial differentials in 
$H_*(RG_i;\partial_i)$. More precisely, $\kappa_i$ is the composition
\[
    H_*(RG_i;\partial_i)=\frac{\mathrm{Ker}\partial_i}{\mathrm{Im}\partial_i}\stackrel{\iota_i}{\longrightarrow} \mathrm{Ker}\partial_i\stackrel{\subset}{\longrightarrow} RG_i
\]
in which $\iota_i$ is a chosen section of the quotient $\mathrm{Ker}\partial_i\to H_*(RG_i;\partial_i)$.
By K\"{u}nneth Theorem, we have an isomorphism 
\begin{equation}
H_*(R(\prod_{k=1}^nG_{i_k});\underline{\partial})=\bigotimes_{i=1}^nH_*(RG_{i_k};\partial_{i_k})\label{iso:H}
\end{equation}
of graded $R$-modules, which is induced by the composition  
\begin{equation}\begin{CD}
  \bigotimes_{k=1}^n H_*(RG_i;\partial_i)@>\underline{\kappa}>> \bigotimes_{k=1}^n(RG_{i_k};\partial_{i_k})@>\nabla>\simeq> (R(\prod_{k=1}^nG_{i_k});\underline{\partial}),\label{def:nabla}
\end{CD}\end{equation}
where the shuffle map $\nabla$ is a chain homotopy inverse of the 
Alexander-Whitney map. Notice that 
\[H_*(RG_i;\partial_i)=R\oplus J_i \]
 as $R$-modules, where $J_i$ is the cokernel of the coaugmentation
$R\to H_*(RG_i;\partial_i)$. Now we specify a linear basis $S_i$ for each coaugmentation ideal $J_i$, 
and consider the diagram
\begin{equation}
\begin{CD}
    A'_w @>i_w>>  \bigotimes_{k=1}^n(R\oplus J_{i_k})@>p_w>>R\oplus\bigotimes_{k=1}^nJ_{i_k}\\
  @V\simeq VV          @V\simeq VV               @V\simeq VV\\
R(G'_w)@>i_{w}>>R(\prod_{k=1}^nG_{i_k})@>p_w>>R(\wedge_{k=1}^n G_{i_k}),
\end{CD}\label{diag:RC}
\end{equation}
in which we identify $A'_w$ under $i_w$ as a $R$-linear subspace, spanned by 
elements of the form $\otimes_{k=1}^nc_{i_k,k}$, $c_{i_k,k}=1$ or 
$s_{i_k,k}\in S_{i_k}$, $k=1,\ldots,n$, with
at least one $c_{i_k,k}=1$. Clearly for such a basis element, 
$p_w(\otimes_{k=1}^nc_{i_k,k})=0$ unless $c_{i_k,k}=1$ for every $k$. 
We see that the Diagram \eqref{diag:RC} of differential graded $R$-modules is commutative, with each vertical 
morphism inducing an isomorphism in homology. 

Consider a filtration on $\underline{A}^K$ with 
$F_n(\underline{A}^K)$ spanned by all basis elements $a=\prod_{k=1}^ls_{i_k,k}$ 
in locally minimal 
presentations (see Theorem \ref{thm:Hopfmain}), with $l\leq n$.  With trivial differentials in $\underline{A}^K$, we define a morphism
$\psi\co\underline{A}^K\to R\underline{G}^K$ of differential graded $R$-modules
as the composition (see \eqref{def:nabla}, \eqref{CD:FG} , 
where we replaced $G_w$ by $RG_w$)
\[\psi(a)=\mu_{w}\circ \nabla\circ \underline{\kappa}(\otimes_{k=1}^ns_{i_k,k}),\]
which is clearly a chain map preserving the filtrations. Consider the spectral sequences with respect to the filtrations $F_n$ on 
$\underline{A}^K$ and $R\underline{G}^K$, respectively. Diagrams \eqref{CD:FG} and \eqref{diag:RC} 
implie that $\psi$ induces an isomorphism on $E^1$-terms, hence on passage to homology, 
$\psi_*\co \underline{A}^K\to H_*(R\underline{G}^K)$ is an isomorphism of graded $R$-modules. 
Finally, to see that $\psi_*$ is actually an isomorphism of 
connected Hopf algebras, from the universal property of $\underline{A}^K$ it suffices to show that whenever restricted to $\sigma\in K$, 
$\psi_*\co \otimes_{i\in\sigma}A_i\to H_*(\underline{G}^K;R)$ is a morphism of connected Hopf algebras, and the diagram
\[
\xymatrix{
	    \bigotimes_{i'\in\sigma'}A_{i'} \ar[rr]^{F(\sigma'\to\sigma)} \ar[rrd]^{\psi_*} &  &   \bigotimes_{i\in\sigma}A_{i}\ar[d]^{\psi_*}\\
	                   &        &      H_*(\underline{G}^K;R)          
	}
\]
commutes. This is clear from the definition of $\psi_*$, and the fact that the isomorphism  \eqref{iso:H} from K\"{u}nneth Theorem is an isomorphism of connected Hopf algebras.

On the naturality, the assumptions imply that $\underline{A}^K\to\underline{\wt{A}}^{\wt{K}}$ and 
$R\underline{G}^K\to R\underline{\wt{G}}^{\wt{K}}$ are both filtration preserving, hence the Diagram 
\eqref{CD:KK} commutes.
\end{proof}
\begin{cor}
Let $\underline{X}^K$ be a polyhedral product associated to a flag complex $K$ and
simply connected and pointed $CW$ complexes $\underline{X}=(X_i)_{i=1}^m$. Then we 
have an isomorphism
\[
   \psi_*\co\underline{A}^K\stackrel{\cong}{\rightarrow} H_*(\Omega\underline{X}^K;R),
\]
when $R$ is a field.
Moreover, this isomorphism is natural with respect to a morphism 
$\underline{X}^K\to\underline{\wt{X}}^{\wt{K}}$ of polyhedral products with respect to a 
simplicial inclusion
$K\to\wt{K}$ of flag complexes with the same vertices, together with morphisms $X_i\to\wt{X}_i$ of simply connected and pointed $CW$ complexes, $i=1,\ldots,m$: the diagram
\begin{equation}
\begin{CD}
                         \underline{A}^K @>\psi_*>\cong> H_*(\Omega\underline{X}^K;R)\\
                         @VVV                              @VVV\\
                         \underline{\wt{A}}^{\wt{K}}@>\psi_*>\cong> H_*(\Omega \underline{\wt{X}}^{\wt{K}};R)\label{CD:KK2}
\end{CD}
\end{equation}
commutes.
\end{cor}
\begin{proof}
Up to homotopy equivalence we replace $X_i$ by $|\overline{W}(G_i)|$ (see \eqref{diag:HC} for more details), 
with $G_i=GSing(X_i)$ the Kan construction of the simplicial set $Sing(X_i)$ of singular triangles. 
The triviality of $\pi_1(X_i)$ implies that $G_i$ is reduced, up to homotopy. 
 The corresponding homotopy equivalence $G\underline{\overline{W}(G)}^K\to \underline{G}^K$ (see \eqref{def:eta}) is natural, giving a commutative diagram
\[
 \begin{CD}
                   \Omega|\underline{\overline{W}(G)}^K|@<<< |G\underline{\overline{W}(G)}^K|@>>>|\underline{G}^K|\\
                             @VVV                                                       @VVV                                                @VVV\\
                 \Omega|\underline{\overline{W}(\wt{G})}^{\wt{K}}| @<<<|G\underline{\overline{W}(\wt{G})}^{\wt{K}}|  
                 @>>>|\underline{\wt{G}}^{\wt{K}}| .          
 \end{CD}
\]
of spaces, with each horizontal morphism a homotopy equivalence. 
Now the statements follow from Theorem \ref{thm:Nat}.
\end{proof}
In particular, let $\underline{A}^K\to \underline{A}^{\Delta}=\otimes_{i=1}^mA_i$ be the morphism of connected and cocommutative Hopf algebras induced by the inclusion $\underline{X}^K\to\underline{X}^{\Delta}=\prod_{i=1}^m X_i$, with $\Delta$ the simplex whose vertices coincide with $K$. From the fibration $F\to\underline{X}^K\to \prod_{i=1}^m X_i$ with $F$ the homotopy fiber of the inclusion, we see that the composition 
\[H_*(\Omega F;R)\to H_*(\Omega \underline{X}^K;R)\to H_*(\Omega \prod_{i=1}^m X_i;R)\]
is trivial, with the first morphism monomorphic and second epimorphic, such that 
\[R\otimes_{H_*(\Omega F;R)}H_*(\Omega \underline{X}^K;R)\cong H_*(\Omega \prod_{i=1}^m X_i;R),\]
by Theorem \ref{thm:dec_loops}. The following proposition follows from \cite[Lemma 3.4]{CMN79}.
\begin{prop}\label{prop:exact}
When $R$ is a field and $K$ is a flag complex, 
we have an isomorphism $H_*(\Omega F;R)\cong \underline{A}^K\square_{\underline{A}^{\Delta}}R$ of connected and cocommutative Hopf algebras, where $\underline{A}^K\square_{\underline{A}^{\Delta}}R$ is the 
cotensor associated to the epimorphism 
\[
\begin{CD}
\underline{A}^K@>>> \underline{A}^{\Delta}\\
@V\psi_* V\cong V                          @V\psi_* V\cong V\\
H_*(\Omega\underline{X}^K;R)@>>>H_*(\Omega\underline{X}^{\Delta};R)
\end{CD} 
\]
in which $F$ is the homotopy fiber of the inclusion $\underline{X}^K\to\underline{X}^{\Delta}$ of polyhedral products
of simply connected and pointed $CW$ complexes $\underline{X}=(X_i)_{i=1}^m$.
\end{prop}

\section{On the loop space of the homotopy fiber}\label{sec:loop}
\subsection{Whitehead's Acyclic Theorem for connected Hopf algebras}
Let  $A=\oplus_{n\geq 0}A_n$ be a connected algebra over a field $k$, with the augmentation 
$\varepsilon\co A\to k$.  The degree 
of a homogeneous element $a\in A$ shall be denoted as $|a|$. 
Recall that the (unnormalized) reduced bar construction $\bar{B}(A)=\oplus_{s\geq 0}A^s$ 
is a simplicial vector space  over $k$, where $A^0=k$ and 
\[A^s=\underbrace{A\otimes A\otimes\cdots\otimes A}_{s \ \text{copies}}.\]
An element $\otimes_{i=1}^s a_i\in A^s$ is usually denoted as 
$[a_1|a_2| \cdots |a_s]$. The face and degeneracy operators are given by (we keep $\varepsilon(a_1)$ and $\varepsilon(a_s)$ to match the degrees)
\begin{align}
  d_i([a_1|\cdots|a_s])&=\begin{cases}
      \varepsilon(a_1) [a_2|\cdots|a_s]  & i=0\\
       (-1)^{|a_1|+\ldots+|a_i|}[a_1|\cdots|a_ia_{i+1}|\cdots|a_s] & 1\leq i\leq s-1\\
     (-1)^{|a_1|+\ldots+|a_{s-1}|}  [a_1|\cdots|a_{s-1}]\varepsilon(a_s) & i=s,
  \end{cases}\\\label{def:d_i}
  s_i([a_1|\cdots|a_s])&=\begin{cases}
  (-1)^{\sum_{j<i}|a_j|}[a_1|\cdots|a_{i-1}|1|a_{i+1}|\cdots|a_s] & 0\leq i\leq s, s\geq 1\\
  [1]  & s=0.
\end{cases}
\end{align}
It turns out that under $\bar{d}=\sum_{i=0}^s(-1)^id_i$, the subspace generated by all degenerate 
simplices is acyclic, and after modulo these degenerate simplices $\bar{B}(A)$ is called the 
normalized reduced bar construction. It is well-known that the homology groups of 
$(\bar{B}(A),\bar{d})$ coincides with the homotopy groups of the simplicial vector space 
$\bar{B}(A)$, using the Moore chain complex.

The left bar complex $B(A)=A\otimes_{\tau}\bar{B}(A)$ gives a free resolution\footnote{Following  \cite{HMS74}, we use the notation $\otimes_{\tau}$ to distinguish its differential with the one in a tensor product.}  
\begin{equation}\begin{CD}
                       \cdots  @>>> B_s(A)@>>> \cdots@>>>  B_1(A) @>>>B_0(A)@>>> k \label{def:BA}
\end{CD}
\end{equation}
of $k$ as left $A$-modules, in which $B_s(A)=A\otimes A^s$, where the differential $d\co B_s(A)\to B_{s-1}(A)$ is given by
\begin{align*}
    d (a[a_1|a_2|\cdots|a_s])&=(-1)^{|a|}aa_1[a_2|\cdots|a_s]+\sum_{j=1}^{s-1}(-1)^{n_j}a[a_1|\cdots|a_ja_{j+1}|\cdots|a_s]\\
    &+(-1)^{n_{s-1}}a[a_1|\cdots|a_{s-1}]\varepsilon(a_s)
\end{align*}
with $n_j=|a|+|a_1|+\cdots+|a_j|+j$. We have $\bar{B}(A)=k\otimes_{A}B(A)$ and 
$\bar{d}=(\mathrm{id}_k,d)$. 

Let $\mathcal{H}_k^c$ be the category of connected Hopf algebras over $k$, and let $\mathcal{S}V_k^c$ be the category of connected simplicial vector spaces over $k$. Given monomorphisms $A_3\to A_i$ in $\mathcal{H}_k^c$, $i=1,2$, let $A=A_1\coprod_{A_3} A_2\in \mathcal{H}_k^c$ be the colimit of the diagram $A_1\leftarrow A_3\rightarrow A_2$ (see Theorem \ref{thm:LS} for more details). 
We define $\bar{B}=\bar{B}(A_1)\coprod_{\bar{B}(A_3)}\bar{B}(A_2)\in \mathcal{S}V_k^c$ as the colimit of the corresponding diagram $\bar{B}(A_1)\leftarrow \bar{B}(A_3)\rightarrow \bar{B}(A_2)$. 
Notice that $\bar{B}\subset\bar{B}(A)$ is a proper simplicial subspace spanned by elements of the form 
$[b_1|\cdots|b_s]$ where either
 $b_i\in A_1$ for all $i$ or $b_i\in A_2$ for all $i$, $1\leq i\leq s$ (we identify $A_3$ as a Hopf subalgebra of $A_i$, $i=1,2$). 
 The following theorem is an analogue of Whitehead's Acyclic Theorem (see Theorem \ref{thm:Whitehead}).
\begin{thm}\label{thm:Whitehead2}
Let $A=A_1\coprod_{A_3} A_2\in \mathcal{H}_k^c$ (resp. $\bar{B}=\bar{B}(A_1)\coprod_{\bar{B}(A_3)}\bar{B}(A_2)\in \mathcal{S}V_k^c$) 
be the colimit of the diagram $A_1\leftarrow A_3\rightarrow A_2$ (resp. $\bar{B}(A_1)\leftarrow \bar{B}(A_3)\rightarrow \bar{B}(A_2)$) with respect to monomorphisms $A_3\to A_i$ in $\mathcal{H}_k^c$, $i=1,2$. 
The complex $B=A\otimes_{\tau}\bar{B}$ gives a free resolution of $k$ as left $A$-modules. The inclusion  $i_{\bar{B}}\co\bar{B}\to\bar{B}(A)$ induces an isomorphism of homology groups.
\end{thm}
\begin{proof}
It suffices to suppose $\bar{B}$ is normalized. First we show that $B$ is acyclic by a contracting homotopy $S\co B_*\to B_{*+1}$, following 
Adams \cite{Ada60}. By Theorem \ref{thm:LS}, $A$ admits a basis of the form $a=xc_{j_2,2}c_{j_1,1}a_3$ with $x$ a word of letters $c_{j_t,t}$ (see \eqref{def:right}). We define 
\[
   S(a[b_1|\cdots |b_s])=\begin{cases}(-1)^{|xc_{j_2,2}|}xc_{j_2,2}[c_{j_1,1}a_3|b_1|\cdots|b_s] & \text{$c_{j_1,1},b_1\in A_i$, $i=1$ or $2$ }\\
   (-1)^{|xc_{j_2,2}c_{j_1,1}|}xc_{j_2,2}c_{j_1,1}[a_3|b_1|\cdots|b_s] & \text{otherwise},
    \end{cases}
\]
and extend $S$ linearly to all elements from $B_*$. It remains to check 
\begin{equation}
Sd+dS=\mathrm{id}-\varepsilon, \label{def:Sd}
\end{equation} 
where $\varepsilon\co B_*\to k$ is the augmentation. In the first case 
\begin{align*}
dS(a[b_1|\cdots |b_s])&=a[b_1|\cdots|b_s]+ (-1)^{|c_{j_1,1}a_3|+1}xc_{j_2,2}[c_{j_1,1}a_3b_1|\cdots|b_s]\\
&+\sum_{t=1}^{s-1}(-1)^{n'_t}xc_{j_2,2}[c_{j_1,1}a_3|\cdots|b_tb_{t+1}|\cdots |b_s] )
\end{align*}
with $n'_t=|c_{j_1,1}a_3|+|b_1|+\cdots+|b_t|+t+1$, and 
\begin{align*}
Sd(a[b_1|\cdots |b_s])=S((-1)^{|a|}ab_1[b_2|\cdots|b_s]+\sum_{k=1}^{s-1}(-1)^{n_t}a[b_1|\cdots|b_tb_{t+1}|\cdots|b_s])
\end{align*}
with $n_t=|a|+|b_1|+\cdots+|b_t|+t$. 
Notice that $ab_1=xc_{j_2,2}c_{j_1,1}a_3b_1$ with $c_{j_1,1}a_3b_1\in A_i$ and $c_{j_2,2}\in A_j$ such that $1\leq i\not=j\leq 2$, we have  $c_{j_1,1}a_3b_1=\sum_l k_lc_{j_1,l}'a_{3,l}'$ with each summand $c_{j_1,l}'a_{3,l}'$ in the form \eqref{def:right}, 
where $k_l\in k$, $c_{j_1,l}'\in A_i$ and $a_{3,l}'\in A_3$. The 
linearity of $S$ implies that 
\[S(ab_1[b_2|\cdots|b_s])=(-1)^{|xc_{j_2,2}|}xc_{j_2,2}[c_{j_1,1}a_3b_1|b_2|\cdots|b_s].\] 
A comparison of $n_t-|xc_{j_2,2}|$ and $n'_t$ shows that \eqref{def:Sd} holds in this case. The verification of the second case is parallel, where we replace $\cdots[c_{j_1,1}a_3\cdots]$ by $\cdots c_{j_1,1}[a_3\cdots]$ and change the coefficients respectively.

For the last statement, consider the diagram  
\[
\begin{CD}
   A@>>> B @>>> \bar{B}\\
   @V=VV  @VVV      @VVV\\
   A@>>> B(A) @>>> \bar{B}(A)
\end{CD}
\]
where all vertical morphisms are inclusions, in which the first and second each induce an isomorphism on passage to homology, 
respectively, since $B$ is acyclic. By the Moore Comparison Theorem (see for example \cite{Zee57}), so does the third, as claimed.
\end{proof}

Let $K$ be a flag complex with $m$ vertices, and let $X_i=\bar{B}(A_i)\in\mathcal{S}V_k^c$ with 
$A_i\in\mathcal{H}^c_k$, $i=1,\ldots,m$. Recall that in a cartesian product 
$X\times X'$, $X,X'\in \mathcal{S}V_k^c$ 
the linear space of $n$-simplices is given by the tensor product
$(X\times X')_n=X_n\otimes X_n'$ with the face maps acting diagonally, namely 
$d_j(x,x')=(d_jx,d_jx')$, $j=0,\ldots,n$, for $(x,x')\in (X\times X')_n$.
The polyhedral product $\underline{\bar{B}(A)}^K\in\mathcal{S}V_k^c$ 
is the colimit of the diagram $F\co \mathcal{K}\to \mathcal{S}V_k^c$, with 
\[F(\sigma)=\prod_{i\in\sigma}X_i,\]
in which $F(\emptyset)=k$ is the initial object in $\mathcal{S}V_k^c$. 
More explicitly, up to isomorphism we identify 
$\underline{\bar{B}(A)}^K\subset \prod_{i=1}^m X_i$ as a subspace spanned by 
$(x_i)_{i=1}^m\in \prod_{i=1}^m X_i$, such that the collection of $i$, with $x_i$ not a vector from $k=\bar{B}_0(A_i)$  by degeneracy maps, spans a simplex of $K$.  
\begin{thm}\label{thm:isophi}
Let $K$ be a flag complex with $m$ vertices and $A_i\in\mathcal{H}^c_k$, $i=1,\ldots,m$. The morphism $\phi\co\underline{\bar{B}(A)}^K\to \bar{B}(\underline{A}^K)$ in $\mathcal{S}V^c_k$ that sends $(x_i)_{i=1}^m$, 
$x_i=[a_i^1|\cdots|a_i^s]$, to 
\[(-1)^{\mathrm{Shuf}(x_1,\ldots, x_m)}[a_1^1a_2^1\cdots a_m^1|\cdots|a_1^sa_2^s\cdots a_m^s]\]
where $\mathrm{Shuf}(x_1,\ldots, x_m)$ is the shuffle number obtained by changing graded vector
\begin{equation}
a^1_1\otimes\cdots\otimes a^s_1\otimes a^1_2\otimes\cdots\otimes a^s_2\otimes\cdots\otimes a^1_m\otimes\cdots\otimes a^s_m \label{def:ori}
\end{equation}
into
\begin{equation}
a^1_1\otimes\cdots\otimes a^1_m\otimes a^2_1\otimes\cdots\otimes a^2_m\otimes\cdots\otimes a^s_1\otimes\cdots\otimes a^s_m, \label{def:aft}
\end{equation}
is well-defined and induces a weak homotopy equivalence.
\end{thm}  
\begin{proof}
First we check that $\phi$ preserves the face and degeneracy maps. Suppose $(x_i)_{i=1}^m\in \bar{B}_s(A_i)$. 
By definition 
$d_0(x_i)_{i=1}^m=(d_0x_i)_{i=1}^m$, which is a new vector obtained via replacing $a^1_i$ by $\varepsilon(a^1_i)$ in \eqref{def:ori}, $i=1,\ldots,m$. Since $\varepsilon$ is multiplicative and keeps the degrees of elements, $\mathrm{Shuf}(x_1,\ldots, x_m)$ remains the same after we change the form of the new vector from \eqref{def:ori} to \eqref{def:aft}, whence $\phi d_0=d_0\phi$. Similarly $\phi d_s=d_s\phi$ and $\phi s_j=s_j\phi$, $j=0,\ldots,s$. Next we check that
\[d_1(x_i)_{i=1}^m=(-1)^{\sum_{t}|a^t_1|}a^1_1a^2_1\otimes\cdots \otimes a^s_1\otimes a^1_2a^2_2\otimes\cdots \otimes a^s_2\otimes\cdots \otimes a^1_ma^2_m\otimes\cdots\otimes a^s_m,\]
hence 
\[\phi(d_1(x_i)_{i=1}^m)=(-1)^{\mathrm{Shuf}(d_0(x_i)_{i=1}^m)} [a^1_1a^2_1a^1_2a^2_2\cdots a^1_ma^2_m|a_1^3a^3_2\cdots a^3_m|\cdots|a_1^sa_2^s\cdots a_m^s].\]
On the other hand, in $\underline{A}^K$ we have
\[
  a^1_1a^2_1a^1_2a^2_2\cdots a^1_ma^2_m=(-1)^{\mathrm{Shuf}(\bm{a}^1,\bm{a}^2)}a^1_1\cdots a^1_ma^2_1\cdots a^2_m
\]
where $\bm{a}^i=(a^i_t)_{t=1}^m$, $i=1,2$, since $\{i\mid \sum_{t}|a_t^i|>0\}$ is a simplex of $K$, by the construction of 
$\underline{\bar{B}(A)}^K$. We see that   $\phi d_1=d_1\phi$ follows from the equation
\[
  \mathrm{Shuf}(d_0(x_i)_{i=1}^m)=\mathrm{Shuf}(\bm{a}^1,\bm{a}^2)+\mathrm{Shuf}((x_i)_{i=1}^m) \quad \mathrm{mod} \ 2.
\]
The verification of $\phi d_j=d_j\phi$, $j=2,\ldots,s-1$, is similar.

Consider the case when $K=\Delta$ is a simplex 
 with $m$ vertices. Now $\underline{A}^K=\otimes_{i=1}^mA_i$, the 
         graded tensor product, and $\underline{\bar{B}(A)}^K=\prod_{i=1}^m\bar{B}(A_i)$.
Clearly $\phi\co\prod_{i=1}^m\bar{B}(A_i)\to \bar{B}(\otimes_{i=1}^mA_i)$ is an isomorphism 
of simplicial vector spaces. 
		
	In general we use an induction on the number $m$ of vertices of $K$, which is trivial when $m=1$. 
	Suppose $\phi$ is a weak homotopy equivalence for every simplicial complex whose number of vertices 
	is smaller than $n\geq 2$, and let $K$ be a flag complex with $n$ vertices.
	By Lemma \ref{lem:L11}, either $K$ is a simplex which is already proved as above, 
	or $K=K_1\coprod_{K_3} K_2$ is the union of two
	proper flag complexes $K_1, K_2$ along  their intersection $K_3$, which is also flag. 
	From the diagram
	\[\xymatrix{
	                                \underline{\bar{B}(A)}^{K_1}\coprod_{\underline{\bar{B}(A)}^{K_3}}\underline{\bar{B}(A)}^{K_2}   \ar[rd]  \ar[d]_{\phi_1\coprod_{\phi_3}\phi_2} & \\
	                                  \overline{B}(\underline{A}^{K_1})
	    \coprod_{\bar{B}(\underline{A}^{K_3})}\bar{B}(\underline{A}^{K_2}) \ar[r]^{i_{\bar{B}}}   &
	    \bar{B}(\underline{A}^{K_1}\coprod_{\underline{A}^{K_3}}\underline{A}^{K_2})
	}
	   	\]
		of simplicial inclusions with $\phi_i$ a weak homotopy equivalence 
		by induction hypothesis, $i=1,2,3$, we see that $\phi_1\coprod_{\phi_3}\phi_2$ is also a weak homotopy equivalence by the Gluing Lemma. 
		We finish the induction by Theorem \ref{thm:Whitehead2}, together with the isomorphism 
		$ \underline{A}^{K_1}\coprod_{\underline{A}^{K_3}}\underline{A}^{K_2}=\underline{A}^K$ by the uniqueness of
		colimits in $\mathcal{H}^c_k$ (see Lemma \ref{lem:decA}).  
\end{proof}
\subsection{Polyhedral products of graded chain complexes}
Let $\mathcal{C}_k^c$ be the category of connected chain complexes over $k$. 
For two objects $(C,d_c),(C',d_{C'})\in\mathcal{C}_k^c$, we have
$ (C\otimes C')_n=\oplus_{k=0}^n C_k\otimes C_{n-k}$
and 
\begin{equation}
	d_{C\otimes C' }(c\otimes c')=d_c(c)\otimes c'+(-1)^{|c|}c\otimes 
	d_{C'}(c')\label{def:tensord}
\end{equation} 
where $c\in C$ (resp. $c'\in C'$) is a homogeneous element of degree $|c|$ (resp. $|c'|$). 

Let $\prod_{i=1}^m X_i$ be a product of connected 
simplicial vector spaces. We have the corresponding simplicial chain complex  
$(\prod_{i=1}^mX_i,d)$, where
$d=\sum_{i}(-1)^{i-1}d_i$. 
The Alexander-Whitney map 
\[
f_{AW}\co (\prod_{i=1}^mX_i,d)\to \otimes_{i=1}^m(X_i,d)
\]
induces a chain homotopy equivalence. 
In particular, when $(X_i,d)=(\bar{B}(A_i),\bar{d})$, $i=1,\ldots,m$, 
consider the polyhedral tensor product 
\[\underline{(\bar{B}(A),\bar{d})}^K\subset\otimes_{i=1}^m(\bar{B}(A_i),\bar{d})\] 
spanned by elements of the form $\otimes_{i=1}^m\bar{c}_i^{s_i}$, $c_i^{s_i}\in \bar{B}_{s_i}(A_i)$, such 
that 
\[\{i\mid c_i^{s^i} \text{is not from } \bar{B}_0(A_i) \text{ by degeneracy maps}\}\in K\] 
as a simplex.
Once we identify the polyhedral product $\underline{X}^K\subset \prod_{i=1}^mX_i$ 
as a subcomplex, its image 
under $f_{AW}(\underline{X}^K)\subset \otimes_{i=1}^m(X_i,d)$ coincides with 
$\underline{(X,d)}^K$. It can be checked that when restricted to $\underline{X}^K$,
$f_{AW}$ induces a chain homotopy equivalence onto its image. 

Now we define an object $C=\oplus_{s\geq 0}C_s=
\underline{A}^K\otimes_{\tau}\underline{(\bar{B}(A),\bar{d})}^K$ in $\mathcal{C}_k^c$, 
where $C_s$ is spanned by 
\[
	a\otimes\otimes_{i=1}^m \bar{c}_i^{s_i}, 
\quad a\in\underline{A}^K, \quad \bar{c}_i^{s_i}=[a_i^1|\cdots|a_i^{s_i}]\in\bar{B}_{s_i}(A_i),\]
such that $\sum_is_i=s$. 
The differential $d_{\tau}$ satisfies
\begin{equation}
	d_{\tau}(a\otimes\otimes_{i=1}^m\bar{c}_i)=(-1)^{|a|}a\left[d'(\otimes_{i=1}^m \bar{c}_i)+
	\bar{d}(\otimes_{i=1}^m \bar{c}_i)\right],\label{def:dtau2}
\end{equation}
where \[d'(\otimes_{i=1}^m \bar{c}_i)=
	\sum_{i=1}^m (-1)^{n_i}a_i^{1}\bar{c}_1\otimes\cdots\otimes
	\bar{c}_{i-1}\otimes
[a_i^2|\cdots|a_i^{s_i}]\otimes \bar{c}_{i+1}\otimes\cdots\otimes \bar{c}_m, \]
$n_i=(\sum_{j<i}|\bar{c}_j|)(|a_i^1|+1)$, 
	and $\bar{d}=\bar{d}_{\otimes_{i=1}^m\bar{B}(A_i)}$ 
	is the tensor product of differentials under the rule 
	\eqref{def:tensord}. It can be checked directly that
$(\underline{A}^K\otimes_{\tau}\underline{(\bar{B}(A),\bar{d})}^K,d_{\tau})$ is a chain 
complex.\footnote{Here we follow the construction of $d_{\tau}$ in \cite{HMS74}.} 
	Moreover, with the induced 
	differential, $k\otimes_{\underline{A}^K}
	\underline{A}^K\otimes_{\tau}\underline{(\bar{B}(A),d)}^K$ coincides with 
	$\underline{(\bar{B}(A),\bar{d})}^K$.

	\begin{prop}\label{prop:acyclic}
	Let $K$ be a flag complex with $m$ vertices, and let 
	$\underline{A}=(A_i)_{i=1}^m$ be $m$ objects in $\mathcal{H}_k^c$.
	The chain complex $(\underline{A}^K\otimes_{\tau}\underline{(\bar{B}(A),
	\bar{d})}^K,d_{\tau})$ is a free resolution of $k$, as left $\underline{A}^K$
	modules.
\end{prop}
\begin{proof}
	It suffices to show that the chain complex
	$(\underline{A}^K\otimes_{\tau}\underline{(\bar{B}(A),
	\bar{d})}^K,d_{\tau})$ is acyclic.
	By definition $B(\underline{A}^K)=\underline{A}^K\otimes_{\tau}\bar{B}(\underline{A}^K)$
	is acyclic. After checking the differentials, 
	we have a commutative diagram
	\[\begin{CD}
	\underline{A}^K@>>>
	\underline{A}^K\otimes_{\tau}\underline{(\bar{B}(A),\bar{d})}^K
		@>>>\underline{(\bar{B}(A),\bar{d})}^K \\
	@A\mathrm{id}_{\underline{A}^K}A=A   @A\mathrm{id}_{\underline{A}^K}\otimes f_{AW}AA 
	@Af_{AW}AA\\
	\underline{A}^K@>>>\underline{A}^K\otimes_{\tau}\underline{\bar{B}(A)}^K
	@>>>\underline{\bar{B}(A)}^K\\
	@V\mathrm{id}_{\underline{A}^K}V=V @V\mathrm{id}_{\underline{A}^K}\otimes\phi VV @V\phi VV\\
	\underline{A}^K@>>> \underline{A}^K\otimes_{\tau}\bar{B}(\underline{A}^K)@>>>
	\bar{B}(\underline{A}^K),
	\end{CD}
\]
in which $\phi$ induces a chain homotopy equivalence (see Theorem \ref{thm:isophi}), so does 
the Alexander-Whitney map $f_{AW}$. We see that the conclusion follows from the Moore Comparison
Theorem.
\end{proof}
Let $\underline{(B(A),d)}^K\subset\otimes_{i=1}^m(B(A_i),d)$ be the polyhedral product with respect 
to $K$ and bar resolutions $(B(A_i),d)=A_i\otimes_{\tau}(\bar{B}(A_i),\bar{d})$ of $A_i$, 
$i=1,\ldots,m$. By definition, $\underline{(B(A),d)}^K$ is spanned by 
\begin{equation}
	\otimes_{i=1}^m a_i^0\bar{c}_i, \quad a_i^*\in A_i,\ \bar{c}_i\in \bar{B}_{s_i}(A_i),\label{def:BAK}
\end{equation}
such that $\{i\mid s_i>0\}\in K$.
\begin{prop}\label{prop:poly}
	Let $\underline{A}^K\to\underline{A}^{\Delta}$ be the epimorphism induced from the inclusion
	$K\to \Delta$ of $K$ in the simplex with the same vertices, where 
	$\underline{A}=(A_i)_{i=1}^m$ are $m$ objects in $\mathcal{H}_k^c$.
The chain complex 
\[\underline{A}^{\Delta}\otimes_{\underline{A}^K}\underline{A}^K\otimes_{\tau}
	\underline{(\bar{B}(A),\bar{d})}^K
\cong \underline{A}^{\Delta}\otimes_{\tau} \underline{(\bar{B}(A),\bar{d})}^K \] 
is isomorphic to the polyhedral product
	$\underline{(B(A),d)}^K$.
\end{prop}
\begin{proof}
	We define a linear map 
            $f\co \underline{A}^{\Delta}\otimes_{\tau} \underline{(\bar{B}(A),\bar{d})}^K 
            \to\underline{(B(A),d)}^K$
	such that 
	\[f(a_1^0\cdots a_m^0\otimes \otimes_{i=1}^m\bar{c}_i)=
		(-1)^{\mathrm{Shuf}(\bm{a}_0,\bm{c} )}
		\otimes_{i=1}^ma_i^0\bar{c}_i,\quad
		\bar{c}_i=[a_i^1|\cdots|a_i^{s_i}], 
	\]
	where $\mathrm{Shuf}(\bm{a}_0,\bm{\bar{c}})$ is the sign obtained by 
	changing the graded vector
	\[a_1^0\otimes\cdots\otimes a_m^0\otimes \bar{c}_1\otimes\cdots\cdots \bar{c}_m\]	
		into $\otimes_{i=1}^ma_i^0\otimes \bar{c}_i$, with $|a_i^0|$ and 
		$|c_i|=\sum_{j=1}^{s_i}|a_i^j|$ the degrees in $A_i$, respectively. 
		It is straightforward to check $f$ is a chain map 
		(see \eqref{def:dtau2},\eqref{def:BA}), using the fact
		$a_i^*a_j^*=(-1)^{|a_i^*||a_j^*|}a_j^*a_i^*$ in $\underline{A}^{\Delta}$,
		whenever $i\not=j$.
\end{proof}

\subsection{The BBCG spectral sequence of \cite{BBCG17}} 
To determine the homology of of the chain complex $\underline{C}^K=
\underline{(B(A),d)}^K$ in Proposition \ref{prop:poly}, we use the spectral sequence
defined in \cite{BBCG17} which is induced by ordering all simplices of $K$ lexicographically. 
First we define $\emptyset$ as the minimum, and for two simplices $\sigma=\{i_1,\ldots,i_s\}$, $\sigma'=\{i'_1,\ldots,i'_{s'}\}$, 
whose indices are written in an increasing order, we define $\sigma'<\sigma$ 
if either $s'<s$ or $s=s'$ and
there exists $1\leq j\leq s$ such that $i'_{j}<i_{j}$ and $i'_t=i_t$ for all $t<j$, if any. In this
way we obtain a filtration
\begin{equation}
	\underline{C}^{F_{\emptyset}K}\subset \underline{C}^{F_{\{1\}}K}\subset\cdots
	\subset\underline{C}^{F_{\sigma}K}\subset\cdots\subset
 \underline{C}^{K} \label{def:fil}
 \end{equation} of chain complexes with $\underline{C}^{F_{\sigma}K}$ spanned by elements 
 of the form \eqref{def:BAK} such that 
 $\{i\mid s_i>0\}\in F_{\sigma}K=\{\sigma'\in K\mid\sigma'\leq\sigma\}$. 
 In the associated spectral sequence we have  
\begin{equation}
E_1^{*,*}=\bigoplus_{\sigma\in K} E_1^{\sigma,*}, \quad
E_1^{\sigma,*}=H_{*}(\underline{C}^{F_{\sigma}K},\underline{C}^{F_{\sigma'}K})\label{E1}
\end{equation}  
 in which $\sigma'$ is the largest simplex smaller than $\sigma$. 
 By definition $F_{\sigma'}K$ contains the subcomplex 
 $\partial \sigma$ as the union of all proper faces of $\sigma$, 
 hence the inclusion
 \begin{equation}
	 (\underline{C}^{F_{\sigma}K}, \underline{C}^{F_{\partial \sigma}K}) \to 
	 (\underline{C}^{F_{\sigma}K},\underline{C}^{F_{\sigma'}K}) \label{iso:ck}
 \end{equation}
 is an isomorphism of chain complexes, where $\underline{C}^{F_{\partial \sigma}K}$ is spanned by 
 elements of the form \eqref{def:BAK} such that $\{i\mid s_i>0\}\in \partial\sigma$. 
  Under the identification $ A_i=A_i\otimes \bar{B}_0(A_i)$ with trivial differentials, 
  we have the isomorphism 
  \[(\underline{C}^{\sigma}, \underline{C}^{\partial \sigma})=\otimes_{i\in\sigma} (B(A_i),A_i)
  \otimes\otimes_{j\not\in\sigma}A_j \]
  of chain complexes.
  The long exact sequence
  \[
   \cdots\rightarrow H_n(A_i)\rightarrow 
   H_n(B(A_i))\rightarrow H_n(B(A_i),A_i)\rightarrow H_{n-1}(A_i)\rightarrow
   \cdots
  \]
together with the acyclicity of $B(A_i)$ imply the isomorphism 
$H_{*+1}(C_i,A_i)\cong \wt{H}_{*}(A_i)=J(A_i)$, with 
$J(A_i)=\mathrm{coker}(k\to A_i)$ the coaugmentation ideal, which we denote by
 $H_{*}(C_i,A_i)\cong  sJ(A_i)$. The conclusion below follows from K\"{u}nneth formula: 
\begin{prop}\label{prop:E1}
We have an isomorphism of graded vector spaces
\[
	E_1^{\sigma,*}(\underline{C}^K)\cong \otimes_{i=1}^mH_*(Y_i), \quad H_*(Y_i)= \begin{cases}
	    H_{*}(B(A_i),A_i)\cong  sJ(A_i) & i\in\sigma\\
    A_i & i\not\in \sigma.
    \end{cases}
\]
\end{prop}
\subsection{A Comparison of Spectral Sequences}
To deal with the higher differentials, we use the method of \cite{BBCG17} to compare 
$\underline{C}^K$ with the cellular chain complex of certain $CW$ complex. 
Now we suppose that as a connected $k$-algebra, each $A_i$ has a homogeneous basis $\mathcal{B}_{A_i}$ containing the identity $1_i$ in dimension $0$ (for our purpose it suffices to suppose each $\mathcal{B}_{A_i}$ is at most countable), $i=1,\ldots,m$. It follows that as a vector space,
$\underline{A}^{\Delta}=\otimes_{i=1}^mA_i$ has a basis consisting of graded monomials of the form 
 \begin{equation}
 \bm{b}=\otimes_{i=1}^mb_i, \quad b_i\in\mathcal{B}_{A_i}. \label{vec:a}
 \end{equation} 
 From the acyclicity of $B(A_i)$, 
 for each $b_i\in \mathcal{B}_{A_i}$ of positive degree we specify 
 $\hat{b}_i\in B_1(A_i)$ so that 
 \begin{equation}
	 d(\hat{b}_i)=(-1)^{|b_i|}b_i. \label{def:hatb}
\end{equation}
According to Proposition \ref{prop:E1}, $E_1^{\sigma,*}(\underline{C}^K)$ admits a basis, 
in which each element has a representative of the form
\begin{equation}
      \otimes_{i=1}^m c_i, \quad c_i=\begin{cases} \hat{b}_i & i\in\sigma\\
      b_i & i\not\in\sigma. \end{cases}
                                                                 \label{basis:E1}
 \end{equation}

Next we consider the minimal cellular chain complex of a disk and its bounding sphere. 
Let $R$ be the chain complex over $k$ spanned by three elements $u$, 
$t$ and $1$ respectively, where $|u|=1$ and $|t|=|1|=0$; 
the differential $d$ in $R$ follows the rule 
$du=t$, $dt=d1=0$. Let $J(R)\subset R$ (resp. $k\subset R$) be the subcomplex spanned by 
$t$ and $u$ (resp. spanned by $1$), so that $R=J(R)\oplus k$. For every natural number 
$n$ we define the suspension $s^nR$ of $R$ spanned by
 $s^nu$, $s^nt$ and $1$, with degrees $n+1$, $n$
and $0$, respectively, such that $ds^nu=(-1)^ns^nt$. 

Let $\bm{n}=(n_i)_{i=1}^m$ be an $m$-tuple of non-negative integers, and let 
\[ I_{\bm{n}}=\{i\mid n_i>0\} \]
be the support of $\bm{n}$.
We define the linear operator 
\[s^{\bm{n}}\co \underline{R}^{K_{I_{\bm{n}}}}\to\otimes_{i=1}^m(s^{n_i}R,d),\] 
where 
\begin{equation}
\underline{R}^{K_{I_{\bm{n}}}}=\otimes_{i=1}^mY_i, \quad 
Y_i=\begin{cases} J(R) & i\in I_{\bm{n}}\\ k  & \text{otherwise},
\end{cases}\label{def:RKI}
\end{equation}  
and 
\begin{equation}
      s^{\bm{n}}(\otimes_{i=1}^mr_i)=(-1)^{\mathrm{Shuf}(\bm{n},\bm{r})}\otimes_{i=1}^ms^{n_i}r_i. \label{def:sRK}
\end{equation}
$\mathrm{Shuf}(\bm{n},\bm{r})=\sum_{i=1}^m \sum_{j>i}n_j |r_i|$. 
Since $s^{n_i}1_i$ never appears when $n_i>0$, the operator $s^{\bm{n}}$ is well defined, sending 
$\underline{R}^{K_{I_{\bm{n}}}}$ ono-to-one onto its image 
$s^{\bm{n}}\underline{R}^{K_{I_{\bm{n}}}}\subset \otimes_{i=1}^m(s^{n_i}R,d)$. 
Moreover, a routine check of differential graded algebras shows that
\begin{equation}
   s^{\bm{n}}d=(-1)^{|\bm{n}|}d s^{\bm{n}}, \label{eq:dsd}
\end{equation}
where $|\bm{n}|=\sum_{i=1}^mn_i$.

For a given basis element $\bm{b}=\otimes_{i=1}^mb_i\in \underline{A}^{\Delta}$ 
of the form \eqref{vec:a}, 
let $\bm{n_b}=(|b_i|)_{i=1}^m$.  
Consider a linear map $f_{\bm{b}}\co s^{\bm{n_b}}R^{K_{I_{\bm{n_b}}}}\to\underline{C}^K$ such that 
\begin{equation}
      f_{\bm{b}}(\otimes_{i=1}^mr_i)=\otimes_{i=1}^m c_i, 
      \quad c_i=\begin{cases} \hat{b}_i & r_i=s^{|b_i|}u_i\\
                                             b_i     &   r_i=s^{|b_i|}t_i\\
                                             1_i      &  r_i=1_i.         
\end{cases}\label{def:fa}
\end{equation}
It can be easily checked that 
$f_{\bm{b}}$ is a chain map preserving the degrees on both sides.
The theorem below is essentially a special case of \cite{BBCG17}, in the language of homology.
\begin{thm}\label{thm:BBCG}
	Let $\mathcal{B}_{\underline{A}^{\Delta}}$ be a basis of 
	$\underline{A}^{\Delta}=\otimes_{i=1}^mA_i$ 
consisting of all monomials $\bm{b}=\otimes_{i=1}^mb_i$, $b_i$ running through a 
basis $\mathcal{B}_{A_i}$ of homogeneous elements of $A_i$. 
Then the chain map 
\[
	f\co \bigoplus_{\bm{b}\in \mathcal{B}_{\underline{A}^{\Delta}}}
	s^{\bm{n_b}}R^{K_{I_{\bm{n_b}}}}\to\underline{C}^K
\]
such that $f=f_{\bm{b}}$ on each direct summand induces an isomorphism in homology groups, 
preserving the degrees.  
\end{thm}
\begin{proof}
Clearly $f$ is an injective chain map. 
It suffices to show that $f$ induces a morphism of spectral sequences, which are isomorphisms 
on $E_1$ terms.
Consider the filtration 
\[
	f^{-1}_{\bm{b}}(\underline{C}^{F_{\emptyset}K})\subset 
 	f^{-1}_{\bm{b}}(\underline{C}^{F_{\{1\}}K})\subset\cdots\subset 
 	f^{-1}_{\bm{b}}(\underline{C}^{F_{\sigma}K})\subset\cdots\subset
 f^{-1}_{\bm{b}}(\underline{C}^{K})\]
on $s^{\bm{n_b}}R^{K_{I_{\bm{n_b}}}}$, which is induced by the filtration \eqref{def:fil} 
from the lexicographic order on simplices of $K$.
It follows that we have a morphism 
\[E_*(f_{\bm{b}})\co E_*^{*,*}( s^{\bm{n_b}}R^{K_{I_{\bm{n_b}}}})\to 
E_*^{*,*}(\underline{C}^{K})\] 
preserving all differentials. Observe that (see \eqref{def:RKI})
\[(f^{-1}_{\bm{b}}(\underline{C}^{\sigma}),f^{-1}_{\bm{b}}(\underline{C}^{\partial \sigma}))
	=\otimes_{i=1}^m Y_i',
\quad Y_i'= \begin{cases} (s^{|b_i|}J(R), k \langle s^{|b_i|}t_i\rangle)  & 
	i\in\sigma\cap I_{\bm{n_b}}\\
                     (k,k) & i\in\sigma\setminus I_{\bm{n_b}}\\
                   k \langle s^{|b_i|}t_i\rangle      & i\in I_{\bm{n_b}}\setminus \sigma\\
                   k  & \text{otherwise},
                   \end{cases}\]
where $k \langle s^{|b_i|}t_i\rangle\subset s^{|b_i|}J(R)$ is the subspace generated by 
$s^{|b_i|}t_i$.
 The isomorphism  \eqref{iso:ck} and Proposition \ref{prop:E1} hold for 
 $E_1^{\sigma, *}(s^{\bm{n_b}}R^{K_{I_{\bm{n_b}}}})$ as well, whence
    \[E_1^{\sigma, *}(s^{\bm{n_b}}R^{K_{I_{\bm{n_b}}}})\cong \otimes_{i=1}^mH_*(Y_i').
    \] 
    We see that $E_1^{\sigma, *}(s^{\bm{n_b}}R^{K_{I_{\bm{n_b}}}})$ is non-trivial only 
    when $\sigma\subset I_{\bm{b}}$.
    In this case 
    $H_*(Y_i')$ is spanned by the class represented by 
    $\otimes_{i=1}^m r_i$, where $r_i=s^{|b_i|}u_i, s^{|b_i|}t_i$ or $1_i$ according to
    $i\in \sigma$, $i\in I_{\bm{a}}\setminus\sigma$ or $i\not\in I_{\bm{a}}$, respectively, 
    whose image under $E_1^*(f_{\bm{b}})$     
    is the class represented by (see \eqref{def:fa})
    \[
         \otimes_{i=1}^m c_i, \quad c_i=\begin{cases} \hat{b}_i  & i\in\sigma\\
                   b_i      & i\in I_{\bm{b}}\setminus\sigma\\
                   1_i  & i\not\in I_{\bm{b}}.
                   \end{cases}
    \]
     Clearly $E_1^{\sigma,*}(f_{\bm{b}})$ is injective. 
     Moreover, a comparison with a basis element $\otimes_{i=1}^m c_i$ of 
     $E_1^{\sigma,*}(\underline{C}^K)$ (see \eqref{basis:E1}) shows that 
     as $\bm{b}$ runs through $\mathcal{B}_{\underline{A}^{\Delta}}$, the images
     of $E_1^{\sigma,*}(f_{\bm{b}})$, $\sigma\subset I_{\bm{b}}$, 
     give a basis of $E_1^{\sigma,*}(\underline{C}^K)$.  
     The proof is completed after this comparison of the spectral sequences, 
     as claimed.  
    \end{proof}

\subsection{The Homology}
Now we give an explicit formula for 
\[
	H_*(\bigoplus_{\bm{b}\in \mathcal{B}_{\underline{A}^{\Delta}}}
	s^{\bm{n_b}}R^{K_{I_{\bm{n_b}}}})\cong
	\bigoplus_{\bm{b}\in \mathcal{B}_{\underline{A}^{\Delta}}}
	H_*(s^{\bm{n_b}}R^{K_{I_{\bm{n_b}}}}).
\]
By \eqref{eq:dsd}, it suffices to understand $H_*(R^{K_{I_{\bm{n_b}}}})$.  
Let $K_{I}$ be the full subcomplex of $K$ with respect to $I\subset \{1,\ldots,m\}$, 
namely 
\[
     K_{I}=\{\sigma\in K\mid\sigma\subset I\}.
\]
Recall that the simplicial chain complex $C_*(K_{I})=\oplus_{n=0}^{\infty}C_n(K_{I})$ 
is freely generated (as a vector space over $k$) by simplices $\sigma\in K_{I}$ with 
$\dim\sigma=\mathrm{Card}(\sigma)-1$. 
The differential $d \co C_*(K_{I})\to C_*(K_{I})$ satisfies 
\begin{equation}
     d(\sigma)=\begin{cases}\sum_{i\in\sigma}(-1)^{\mathrm{Shuf}(i,\sigma\setminus\{i\})}\sigma\setminus\{i\}
     & \dim\sigma>0\\
     \emptyset & \dim\sigma=0,
     \end{cases}\label{def:dSK}
\end{equation}
where $\mathrm{Shuf}(i,\sigma\setminus\{i\})=\mathrm{Card}\{j\in\sigma\mid j<i\}$ and $\emptyset$ generates the chains 
of dimension $-1$.   
\begin{lem}
	The map $f_I\co C_*(K_{I})\to \underline{R}^{K_{I}}$ that sends $\sigma$ to 
	$u_{\sigma}t_{I\setminus\sigma}=\otimes_{i=1}^mr_i$, 
where $r_i=u_i, t_i$ or $1_i$ if $i\in\sigma, I\setminus\sigma$ or $i\not\in I$, respectively, 
induces an isomorphism of chain complexes. In particular, $f_I$ induces an isomorphism 
\[
           \wt{H}_*(K_I)=H_{*+1}(\underline{R}^{K_{I}})
\]
of homology groups, shifting the dimensions up by one.
\end{lem}
 \begin{proof}
 Clearly $f_I$ induces an isomorphism of vector spaces (see \eqref{def:RKI}). 
 Notice that
 \begin{equation}
 d(\otimes_{i=1}^mr_i)=\sum_{t=1}^m(-1)^{\sum_{j<t}|r_j|}r_1\otimes\cdots\otimes dr_t\otimes\cdots\otimes r_m,
 \label{eq:dr}
 \end{equation}
 where $\sum_{j<t}|r_j|=\mathrm{Card}\{j<t\mid r_j=u_j\}$, since $|u_j|=1$ while $|t_j|=|1_j|=0$. 
 A comparison of \eqref{def:dSK} and \eqref{eq:dr} shows that $f_I(d\sigma)=df_I(\sigma)$, 
 for every simplex $\sigma\in K_I$, as desired. 
 \end{proof}
 Here is a conclusion of all calculations above. Let 
 $\bm{b}=\otimes_{i=1}^mb_i\in\mathcal{B}_{\underline{A}^{K}}$ 
 be a basis element 
 with support $I_{\bm{b}}=I$ (so $b_i=1_i$ when $i\not\in I_{\bm{b}}$). 
 For a simplicial chain 
 $\sigma\in C_{s-1}(K_{I})$ generated by a single simplex, $s\geq 0$, we have 
 $f_{I}(\sigma)=
 u_{\sigma}t_{I\setminus\sigma}\in R^{K_I}$ of degree $s$ ($f_I(\emptyset)=t_{I}$). 
 On the other hand, $f_{\bm{b}}\circ s^{\bm{b}}$ sends 
 $f_I(\sigma)$ to the chain (see \eqref{def:fa}) 
 \begin{equation}
	 (-1)^{\mathrm{Shuf}(\bm{b}, \sigma)} \otimes_{i\in\sigma}\hat{b}_i 
	 \otimes\otimes_{i\not\in \sigma}b_i\in\underline{(B(A),d)}^K,\label{eq:sbsigma}
 \end{equation}
 where $\hat{b}_i=(-1)^{|b_i|}[b_i]$ (see \eqref{def:hatb}), $\otimes_{i\in\sigma}\hat{b}_i 
 \otimes\otimes_{i\not\in\sigma}b_i=\otimes_{i=1}^mc_i$ with $c_i=\hat{b}_i$ or $b_i$ 
 if $i\in\sigma$ or $i\not\in\sigma$, respectively, 
 $\mathrm{Shuf}(\bm{b}, \sigma)=\sum_{i=1}^m|b_i|\mathrm{Card}(\{j\in\sigma\mid j<i\})$. In this
 way simplicial chains in $C_{s-1}(K_I)$ are sent to chains in $\underline{(B(A),d)}^K$ of 
 simplicial degree $s$.
 Together with Theorem \ref{thm:BBCG}, as $\bm{b}$ runs through $\mathcal{B}_{\underline{A}^K}$, 
 we get all elements in the homology of $\underline{(B(A),d)}^K$. 
 \begin{thm}\label{thm:basis}
	Let $\underline{A}^K\to\underline{A}^{\Delta}$ be the 
	morphism of connected Hopf algebras over a field $k$, which is 
	induced by the inclusion $K\to \Delta$ of the flag complex $K$ into the simplex 
	$\Delta$ with the same vertices. For every $s\geq 1$ 
	we have degree-preserving isomorphisms
	\begin{equation}
	Tor^{\underline{A}^K}_{s,*}(\underline{A}^{\Delta},k)=
	\bigoplus_{I \subset [m]}\wt{H}_{s-1}(K_I;k)\otimes\otimes_{i\in I}J(A_i)=
	\bigoplus_{\bm{b}\in \mathcal{B}_{\underline{A}^{\Delta}}}\wt{H}_{s-1}(K_{I_{\bm{b}}};k)
	\label{iso:Tor}
\end{equation} 
of left $\underline{A}^{\Delta}$-modules, where 
$\mathcal{B}_{\underline{A}^{\Delta}}$ consists of all monomials $\bm{b}=\otimes_{i=1}^mb_i$, $b_i$ running through a 
basis $\mathcal{B}_{A_i}$ of homogeneous elements of $A_i$, $J(A_i)$ the (co)augmentation ideal of $A_i$, 
$i=1,\ldots,m$. 
Moreover, under the isomorphisms above the left 
$\underline{A}^{\Delta}$-action on the left-hand side is given as follows. 
For $\bm{b'},\bm{b}\in \mathcal{B}_{\underline{A}^{\Delta}}$ with 
\begin{equation}
	\bm{b'b}=\sum_{t}k^t\bm{b}''_t,  \quad 0\not=k^t\in k, \ 
	\bm{b}''_t\in \mathcal{B}_{\underline{A}^{\Delta}},\label{eq:b'b}
\end{equation} 
we have  
\[ \bm{b}'\cdot \wt{H}_*(K_{I_{\bm{b}}},k)\subset \bigoplus_{t} \wt{H}_*(K_{I_{\bm{b}''_t}},k), \]
sending $\alpha\in \wt{H}_*(K_{I_{\bm{b}}},k) $ to $\sum_{t}k_t i_*^t(\alpha)$, where $i_*$ is 
induced by the simplicial inclusion $K_{I_{\bm{b}}}\to K_{I_{\bm{b}''_t}}$.
\end{thm}
 Our proof of this theorem requires a standard technique in homological algebra:
 \begin{lem}\label{lem:basis}
 Suppose $g,h\co (C,d)\to\otimes_{i=1}^m(C_i,d_i)$ are two chain maps such that 
 for every additive generator $c\in C$ there exists a corresponding $\otimes_{i=1}^m z_i\in 
 \otimes_{i=1}^mC_i$ with
 $g(c)=\otimes_{i=1}^mc_i$, $h(c)=\otimes_{i=1}^m (c_i+d_iz_i)$. Then
 $g,h$ coincide on passage to homology. 
 \end{lem}
 \begin{proof}
	 Let $g_1$ be the chain map given by $g_1(c)=(c_1+dz_1)\otimes\otimes_{i=2}^m c_i$ 
	 for each additive generator $c\in C$.
	 First we show that there exists a chain homotopy $S$ between $g$ and $g_1$.
	 We define $S(c)=z_1\otimes\otimes_{i=2}^m c_i$, and it is easy to check 
	 \[
		 d_{\otimes_{i}C_i}S(c)+Sd(c)=g_1(c)-g(c),
	 \]
	 for every additive generator $c\in C$, hence the conclusion holds when $h=g_1$. 
	 Similarly we construct $g_k(c)=\otimes_{i\leq k}(c_i+d_iz_i)\otimes(\otimes_{i>k}c_i)$, 
	 $k=2,\ldots,m$, and there is a chain homotopy between $g_k$ and $g_{k+1}$. Finally 
	 we get the chain homotopy between $g$ and $g_m=h$.
 \end{proof}

\begin{proof}[Proof of Theorem \ref{thm:basis}]
	The isomorphism \[Tor^{\underline{A}^K}_{s,*}(\underline{A}^{\Delta},k)=
	H_s\left(\underline{(B(A),d)}^K  \right)\]  follows from Propositions 
	\ref{prop:acyclic}, \ref{prop:poly}, whence the isomorphism \eqref{iso:Tor}, 
	by the conclusion above. On the left action of $\underline{A}^{\Delta}$, 
	suppose $\bm{b}=\otimes_{i=1}^m b_i, \bm{b'}=\otimes_{i=1}^m b_i'$, 
	with $b_i,b_i'$ from a basis $\mathcal{B}_{A_i}$ of $A_i$, 
	such that $b_i'b_i=\sum_{t_i} k_i^{t_i}b''_{it_i}$ as a linear sum of basis elements in 
	$\mathcal{B}_{A_i}$, $i=1,\ldots,m$. Clearly in \eqref{eq:b'b} we have 
	$\bm{b}''_t=\otimes_{i=1}^mb''_{it_i}$ and 
	\begin{equation}
		k^t=(-1)^{\mathrm{Shuf}(\bm{b'},\bm{b})}\prod_{i}k_i^{t_i},\quad 
		t=(t_i)_{i=1}^m,\label{eq:kt}
\end{equation}
hence $k^t\not=0$ 
	already implies that $I_{\bm{b}}\subset I_{\bm{b}_t''}$. Let 
	$\sigma\in C_{s-1}(K_{I_{\bm{b}}})$ be the simplicial chain generated by a single 
	simplex $\sigma\in K_{I_{\bm{b}}}$. Using chains of the form
	\eqref{eq:sbsigma}, it suffices to show that, the two chain maps
	$g,h\co s^{\bm{b}}
		\underline{R}^{K_{I_{\bm{b}}}}\to \underline{(B(A),d)}^K$ with
	\begin{equation}
	g(\sigma)=\bm{b}'\cdot (-1)^{\mathrm{Shuf}(\bm{b}, \sigma)} \otimes_{i\in\sigma}\hat{b}_i 
	\otimes\otimes_{i\not\in\sigma}b_i, \
	h(\sigma)=\sum_{t}k^t(-1)^{\mathrm{Shuf}(\bm{b}''_t, \sigma)} 
	\otimes_{i\in\sigma}\hat{b}_{it_i}'' 
	 \otimes\otimes_{i\not\in\sigma}b_{it_i}''\label{def:homotopy}
 \end{equation}
for every $\sigma\in K_{I_{\bm{b}}}$, coincides on passage to homology,
since each summand in $h(\sigma)$ is  
	$f_{\bm{b}''_t}\circ s^{\bm{b}''_t}\circ f_{I_{\bm{b}_t''}}(\sigma)$, with 
	$\sigma\in C_{s-1}(K_{I_{\bm{b}_t''}})$ the image of the chain map induced by the simplicial
	inclusion $K_{I_{\bm{b}}}\to K_{I_{\bm{b}_t''}}$. In 
	$\underline{A}^{\Delta}\otimes_{\tau}\underline{(\bar{B}(A),\bar{d})}^K\cong 
	\underline{(B(A),d)}^K$ (see Proposition \ref{prop:poly}) we have 
	\begin{align*}
	&\bm{b}'\cdot \otimes_{i\in\sigma}\hat{b}_i 
		\otimes\otimes_{i\not\in\sigma}b_i=
	(-1)^{\mathrm{Shuf}(\bm{b}',\bm{b})+\mathrm{Shuf}(\bm{b}',\sigma)}
	\otimes_{i\in\sigma}b'_i\hat{b}_i\otimes\otimes_{i\not\in\sigma}b_i'b_i\\
   &=(-1)^{\mathrm{Shuf}(\bm{b}',\bm{b})+\mathrm{Shuf}(\bm{b}',\sigma)}
   \otimes_{i\in\sigma} (\widehat{b_i'b_i}+dz_i)\otimes\otimes_{i\not\in\sigma}
   b_i'b_i;\\
&(-1)^{\mathrm{Shuf}(\bm{b}',\bm{b})+\mathrm{Shuf}(\bm{b}',\sigma)}
\otimes_{i\in\sigma} \widehat{b_i'b_i}\otimes\otimes_{i\not\in\sigma}
  b_i'b_i\\
  &=(-1)^{\mathrm{Shuf}(\bm{b}',\bm{b})+\mathrm{Shuf}(\bm{b}',\sigma)}
\otimes_{i\in\sigma} \sum_{t_i}k^{t_i}_i\hat{b}_{it_i}''\otimes\otimes_{i\not\in\sigma}
  (\sum_{t_i}k^{t_i}_ib_{it_i}'')\\
  &=(-1)^{\mathrm{Shuf}(\bm{b},\sigma)}\sum_{t}(-1)^{\mathrm{Shuf}(\bm{b}_t'',\sigma)}k^t
  \otimes_{i\in\sigma} \hat{b}_{it_i}''\otimes\otimes_{i\not\in\sigma}
  b_{it_i}'',
	\end{align*}
	where $d{z}_i=b'_i\hat{b}_i-\sum_{t_i} k^{t_i}_i\hat{b}_{ti}''$, for some 
	${z}_i\in B(A_i)$, which is well defined by the acyclicity of $B(A_i)$
	\footnote{More explicitly, 
		given $\hat{b}_i=(-1)^{|b_i|}[b_i]$ and 
		$\sum_t k^t_i\hat{b}_{ti}''=(-1)^{|b_i'|+|b_i|}[b_i'b_i]$,  
		we can choose 
	$\hat{z}_i=(-1)^{|b_i'|+|b_i|}[b_i'|b_i]$.}, 
	since
	\[d(b'_i\hat{b}_i)=(-1)^{|b_i'|+|b_i|}b_i'b_i=
	d( \widehat{b_i'b_i} ).\]
	We see that $g,h$ induces the same morphism of homology groups, by Lemma \ref{lem:basis}
	(in which we choose $z_i=0$ for $b_i'b_i$).
 \end{proof}

\begin{remark} A similar spectral sequence with easier calculations works for the chain complex
\[k\otimes_{\underline{A}^{\Delta}}\underline{(B(A),d)}^K\cong\underline{(\bar{B}(A),\bar{d})}^K.\]
Notice that that $H_s(\underline{(\bar{B}(A),\bar{d})}^K)\cong Tor^{\underline{A}^K}_{s,*}(k,k)$,\footnote{We 
adopt the notation that an element $[a_1|\cdots|a_s]$ is bigraded with degree $(s,\sum_{i=1}^s|a_i|)$.}
by Theorem \ref{thm:isophi}, and 
\begin{equation*}
H_s(\bar{B}(A_i),\bar{d}_i)\cong Tor^{A_i}_{s,*}(k,k), \quad i=1,\ldots,m.
\end{equation*}
We have an isomorphism 
\begin{equation}
H_{*}(\underline{(\bar{B}(A),\bar{d})}^K)=\underline{H_*(\bar{B}(A),\bar{d})}^K \label{iso:TorAK}
\end{equation}
of bigraded vector spaces (the polyhedral products above are in the category of graded chain complexes). This is an analogue of \cite{BBCG10}[Theorem 2.34, pp. 14--15].
\end{remark}

In what follows we use the same notation as in Theorem \ref{thm:basis}.
 Let $A'=\underline{A}^{K}\square_{\underline{A}^{\Delta}}k$ be the cotensor with respect to 
 the epimorphism $\underline{A}^K\to \underline{A}^{\Delta}$. A theorem of Milnor and Moore shows  
 that 
 \begin{equation}
	 \underline{A}^K=A'\otimes \underline{A}^{\Delta}\label{eq:tensor}
 \end{equation}as both a left $A'$-module and a right 
 $\underline{A}^{\Delta}$ comodule (see \cite[Theorem 4.7]{MM65}). The inclusion 
 $B(A')=A'\otimes_{\tau}\bar{B}(A')\to B(\underline{A}^K)=
 A\otimes_{\tau}\bar{B}(\underline{A}^K)$ of acyclic complexes 
 corresponding to the Hopf subalgebra $A'\subset  \underline{A}^K$
 induces a chain homotopy equivalence, whence the chain homotopy equivalence
 \begin{equation}
	 k\otimes_{A'}A'\otimes_{\tau}\bar{B}(A')\to k\otimes_{A'}\underline{A}^K
	 \otimes_{\tau}\bar{B}(\underline{A}^K)=
	 \underline{A}^{\Delta}\otimes_{\tau}\bar{B}(\underline{A}^K),\label{eq:cor}
 \end{equation}
 which induces the change of ring isomorphism 
      \begin{equation}
      Tor^{A'}_{n,*}(k,k)\cong Tor^{\underline{A}^K}_{n,*}(\underline{A}^{\Delta},k).  \label{iso:cor}
      \end{equation} 
    Recall that the Euler-Poincar\'{e} series of a 
	 graded vector space $V=\oplus_{n=0}^{\infty}V_n$ over a field $k$ (which is of finite 
	 type so that $\mathrm{dim}V_n$ is finite, for each $n$) 
	 is given by 
	 $P(V;t)=\sum_{n=0}^{\infty}(\dim V_n)t^n$.  
 \begin{cor}\label{cor:EP}
		Suppose that $A_i$ is of finite type, $i=1,\ldots,m$.
		 For the Euler-Poincar\'{e} series we have 
	 \begin{equation}\frac{1}{P(A'; t)}=1+\sum_{I\subset\{1,\ldots,m\}\atop I\not=\emptyset}
		\sum_{n=1}^{\infty}(-1)^n
		\mathrm{dim}\wt{H}_{n-1}(K_{I};k)\prod_{i\in I}P(J(A_i);t)\label{eq:EPF}
\end{equation}
and 
	 \begin{equation}
		 P(\underline{A}^K;t)=
		 P(A'; t)\prod_{i=1}^mP(A_i;t). \label{eq:EPX}
	 \end{equation}
 \end{cor}
 \begin{proof}
	 The first statement follows from the well-known fact that 
	 \begin{equation}
	 \frac{1}{P(A;t)}=\sum_{n=0}^{\infty}(-1)^nP(Tor_{n,*}^A(k,k);t)\label{eq:PAT}
	 \end{equation} 
	 for a graded algebra 
	 $A=\oplus_{n\geq 0}A_n$ of finite type, together with Theorem \ref{thm:basis}, 
	 the change of ring isomorphism \eqref{iso:cor} and 
	 \[       
	 P(\otimes_{i\in I}J(A_i);t)=\prod_{i\in I}P(J(A_i);t).\]
	 The second statement follows from
	 \eqref{eq:tensor}.
 \end{proof}
\begin{exm}
Suppose that $K$ bounds a square, with $\{i,i+1\}\in K$, for mod $4$ integers $i=1,2,3,4$, 
and $\underline{X}=(S^{n_i+1})_{i=1}^4$ are $4$ spheres 
with $n_i\geq 1$. By Proposition \ref{prop:exact} the right-hand side of \eqref{eq:EPF} becomes
\begin{align*}
&1-\frac{t^{n_1+n_3}}{(1-t^{n_1})(1-t^{n_3})}-\frac{t^{n_2+n_4}}{(1-t^{n_2})(1-t^{n_4})}+\prod_{i=1}^4\frac{t^{n_i}}{1-t^{n_i}}\\
&=
\frac{1-\sum_{i=1}^4t^{n_i}+\sum_{i=1}^4t^{n_i+n_{i+1}}}{\prod_{i=1}^4(1-t^{n_i})},
\end{align*}
together with \eqref{eq:EPX} we have
\begin{equation}
P(\underline{A}^K)=\frac{1}{1-\sum_{i=1}^4t^{n_i}+\sum_{i=1}^4t^{n_i+n_{i+1}}}\label{eq:PAT1}.
\end{equation}

On the other hand, according to Adams and Hilton \cite[Theorem 4.3]{AH56}, the Adams-Hilton model over 
$k$ of $\underline{X}^K$ is a free associative 
algebra generated by $8$ elements 
$x_i,y_i$ with differentials $dy_i=[x_i,x_{i+1}]$ for $i=1,\ldots,4$ mod $4$, whose homology gives 
$H_*(\Omega\underline{X}^K;k)=\underline{A}^K$. A theorem of Lemaire (see \cite[Theorem 2.3.8, p. 30]{Lem74}) implies that $Tor_{1,*}^{\underline{A}^K}(k,k)$ is spanned 
by $x_1,\ldots, x_4$ and $Tor_{2,*}^{\underline{A}^K}(k,k)$ is spanned by $dy_1,\ldots, dy_4$, respectively, 
and that $Tor_{s,*}^{\underline{A}^K}(k,k)$
vanishes for $s\geq 3$. This coincides with \eqref{eq:PAT1}. In fact, it is straightforward to check that $\underline{X}^K=(S^{n_1+1}\vee S^{n_3+1})\times (S^{n_2+1}\vee S^{n_4+1})$ by the definition of a polyhedral product.\footnote{We would like to thank the referee to point this out.}
\end{exm}

 \begin{cor}
	 As an algebra, $A'$ is finitely presented if  
	 it is finitely generated.  
 \end{cor}
 \begin{proof}
	 It is well-known that $A'$ finitely generated if and only if 
	 the dimension $b_1'$ of  $Tor^{A'}_{1,*}(k,k)$ is finite, and in this case 
	 $A'$ is finitely presented if and only if the dimension $b_2'$ of 
	 $Tor_{2,*}^{A'}(k,k)$ is finite. Suppose $b_2'$ is infinite. 
	 Since the total number of full subcomplexes 
	 $K_I$ is finite, there exists a full subcomplex
	 $K_I$ such that $\otimes_{i\in I}J(A_i)$ is infinite-dimensional, 
	 by Theorem \ref{thm:basis}. This implies that,  
	 at least one $A_{i_0}$, $i_0\in I$, is infinite-dimensional. Moreover, 
	 $H_1(K_I)$ is non-trivial implies that we can find $j_0\in I$, such that $\{i_0,j_0\}$
	 is not an edge of $K$, namely 
	 $\wt{H}_0(K_{ \{i_0,j_0\} } )\otimes J(A_{i_0})
	 \otimes J(A_{j_0})$ is infinite-dimensional, whence $b_1'$ is infinite.
 \end{proof}
 Let $L_x(y)=xy-(-1)^{|x||y|}yx$ be the Lie
 bracket for homogeneous elements $x,y$.
 \begin{cor}\label{Cor:GPTW16}
	As an algebra, $A'$ is generated by the set
 \[S=\bigcup_{\mathrm{dim} \wt{H}_0(K_I)>0}S_I,\]
 where $S_I$ is a collection (writing $I=\left\{ i_l \right\}_{l=1}^n$ in an increasing 
 order) of iterated Lie brackets of the form
	 \begin{equation}
	   L_{x_{i_1}}\circ L_{x_{i_2}}\circ\ldots\circ L_{x_{i_{t-1}}}\circ 
	   L_{x_{i_{t+1}}}\circ\ldots
	   \circ L_{x_{i_{n}}}(x_{i_{t}}),\label{def:Lie}
	\end{equation}
	such that $\{i_t\}$ is the smallest vertex in a connected component of $K_I$ not containing 
	the vertex $\{i_n\}$ and $x_{i_l}$ runs through a homogeneous basis of $J(A_{i_l})$, $l=1,\ldots,n$.	
	Moreover, $S$ generates $A'$ freely if and only if 
	$H_1(K_I)=0$ for every full subcomplex $K_I$ of $K$.
 \end{cor}
 \begin{proof}
	 It is easy to see that an element of the form \eqref{def:Lie} is primitive in 
	 $\underline{A}^K$ (see \eqref{def:diag}) 
	 whence an element in $A'= A\square_{A''}k$,
	 where $A=\underline{A}^K$ and $A''=\otimes_{i=1}^mA_i$.
	 It is well known that (see for example, \cite[Lemma 3.12]{CMN79}), 
	 the image of the element \eqref{def:Lie}  
	 under the change of ring isomorphism \eqref{eq:cor}, is equivalent to 
	 \begin{equation}
		 x_{i_1}\cdots x_{i_{t-1}}x_{i_{t+1}}\cdots x_{i_{n-1}}
		 [L_{x_{i_{n}}}(x_{i_{t}})]\in A''\otimes_{\tau}\bar{B}_1(A),\label{eq:cor2}
   \end{equation}
   up to a boundary from $A''\otimes_{\tau}\bar{B}_2(A)$. Moreover, in $A''\otimes_{\tau}\bar{B}(A)$ we have  
   \begin{align*}
   [&L_{x_{i_n}}(x_{i_t})]+d((-1)^{|x_{i_n}|}[x_{i_n}|x_{i_t}]-(-1)^{|x_{i_t}|+|x_{i_t}||x_{i_n}|}[x_{i_t}|x_{i_n}])\\
   =&(-1)^{|x_{i_n}|}x_{i_n}[x_{i_t}]-(-1)^{|x_{i_t}|+|x_{i_t}||x_{i_n}|}x_{i_t}[x_{i_n}], 
   \end{align*}
   where the right-hand side is a chain of the form \eqref{eq:sbsigma} associated to a non-trivial generator of
   $\wt{H}_0(K_{I'})$, $I'=\{i_k,i_n\}$. By Theorem \ref{thm:basis}, the left multiplication
   of \[x_{i_1}\cdots x_{i_{t-1}}x_{i_{t+1}}\cdots x_{i_{n-1}}\] sends this generator to a generator
   of $\wt{H}_0(K_{I})$; as $i_t$ runs through the minimal vertex in each connected 
   component of $K_I$ not containing $i_n$, we get a basis of $\wt{H}_0(K_{I})$, giving a basis 
   of $Tor_{1,*}^{A'}(k,k)$. The last statement follows from the 
   fact that the algebra $A'$ is free if and only if $Tor_{2,*}^{A'}(k,k)$ vanishes.
 \end{proof}
  \begin{remark}
 We see that the $1$-skeleton $K^1$ being chordal implies that  
$H_1(K_I)=0$ for every full subcomplex $K_I$ of the flag complex $K$. Conversely, if $H_1(K_I)$ is non-trivial for some full subcomplex 
$K_I$, there exists an embedded simplicial cycle of length $\geq 4$ as a representative, since $K$ (hence $K_I$) is flag, which means that 
$K^1$ is not chordal.   
The last statement of Corollary \ref{Cor:GPTW16} coincides with Corollary \ref{cor:PV}, 
namely $F$ is a wedge sum of suspensions.	

Corollary \ref{Cor:GPTW16} (together with Proposition \ref{prop:exact}) is a generalization of a result of \cite{GPTW16} on the generators of the Pontryagin algebra of the loop space of a moment-angle complex, namely $X_i=\mathbb{C}P^{\infty}$, 
 $i=1,\ldots,m$.	
\end{remark}
\section{Appendix: presentations of an element in a graph product}\label{sec:last}
In this appendix we prove some well-known facts on the structures of polyhedral products in groups and connected Hopf algebras, respectively,
in the form that we need in this paper. 
For more details, the readers are referred to Panov and Veryovkin \cite{PV19}, Hermiller and Meier \cite{HM95} for the first part, and Lemaire \cite{Lem74} for the second part. 

\subsection{On discrete groups}
Throughout this subsection we work in the category of groups.
\begin{prop}\label{prop:preG}
Suppose the $m$ groups $\underline{G}=(G_i)_{i=1}^m$ each have a presentation 
$G_i=\langle S_i\mid R_i\rangle$ generated by $S_i$ and subject to relations $R_i$, $i=1,\ldots,m$.
Then the polyhedral product $\underline{G}^K$ has a presentation
\begin{equation}
      \langle \bigcup_{i=1}^mS_i\mid R_1,\ldots, R_m, (S_i,S_j), \forall \{i,j\}\in K\rangle, \label{def:SR}
\end{equation}
in which $(S_i,S_j)$ is the collection of commutators $s_i^{-1}s_j^{-1}s_is_j$ with $s_i,s_j$ running through all elements of $S_i, S_j$, respectively.  
\end{prop}
\begin{proof}
Let $G$ be the group presented by \eqref{def:SR}. It suffices to show that $G$ satisfies 
the universal property of $\underline{G}^K$.  First we define 
$\iota_{\sigma}\co   \prod_{i_k\in\sigma}G_{i_k}\to G $ with 
$\iota_{\sigma}((g_i)_{i\in\sigma})=\prod_{i\in\sigma}g_i$, where $g_i$ is a word with letters from $S_i\cup S_i^{-1}$. 
Clearly $\iota_{\sigma}$ is well-defined, and the diagram 
\[\xymatrix{
	    \prod_{i'_k\in\tau}G_{i'_k} \ar[rr]^{F(\tau\to\sigma)} \ar[rrd]^{\iota_{\tau}} &  &   \prod_{i_k\in\sigma}G_{i_k}\ar[d]^{\iota_{\sigma}}\\
	                   &        &                G
	}
	\]
commutes for every pair $(\sigma,\tau)$ of simplices from $K$, in which 
$F(\tau\to\sigma)$ sends  $\prod_{i'_k\in\tau}G_{i'_k}$ identically onto $\prod_{i_k\in\sigma}H_{i_k}$, 
	where $H_{i_k}=G_{i_k}$ if $i_k\in\tau$, otherwise $H_{i_k}=\langle 1\rangle$ is the trivial group. 
	
Now suppose $G'$ is any group with morphisms 
$f_{\sigma}\co \prod_{i_k\in\sigma}G_{i_k}\to G'$ such that
\[\xymatrix{
	    \prod_{i'_k\in\tau}G_{i'_k} \ar[rr]^{F(\tau\to\sigma)} \ar[rrd]^{f_{\tau}} &  &   \prod_{i_k\in\sigma}G_{i_k}\ar[d]^{f_{\sigma}}\\
	                   &        &                G'
	}
	\]
	commutes for every pair $(\sigma,\tau)$. We define $u\co G\to G'$ so that $u(g_i)=f_i(g_i)$ for all $g_i\in G_i$ (we identify $G_i$ as a subgroup of $\underline{G}^K$ as usual). In this way $u$ is uniquely determined and the universal property holds, once we show that
	it is well-defined, namely $u(R_i)=u((S_i,S_j))=1$, for all possible $i,j$. This is clear from the two commutative diagrams above.
\end{proof}
\begin{remark}
Given morphisms $G_i\to \wt{G}_i$ of groups, $i=1,\ldots,m$, together with a simplicial 
monomorphism $K\to \wt{K}$ of flag complexes with $\mathrm{Vert}(K)=\mathrm{Vert}(\wt{K})$, the universal property induces a unique morphism $\underline{G}^K\to\underline{\wt{G}}^{\wt{K}}$. With the presentations \eqref{def:SR}, 
this morphism sends $(S_i,S_j)$ to $1\in \underline{\wt{G}}^{\wt{K}}$, if $\{i,j\}$ spans an edge in $\wt{K}$.
\end{remark}
With the theorem above we can write $g\in\underline{G}^K$ with $g\not=1$ in a word $g=\prod_{k=1}^ng_{i_k,k}$ with $1\not=g_{i_k,k}\in G_{i_k}$. Such 
a word is locally minimal if any single 
exchange of two adjacent letters, if possible, will make it larger in the lexicographic order. 

Recall that the structure theorem (see \cite{Ser02}) of  
the amalgamated sum $H=H_1*_{H_3}H_2$ of groups $H_1,H_2$, with 
$H_3\subset H_i$ identified as a subgroup, $i=1,2$, states that every $h\in H$ is written uniquely
as
\begin{equation}
h=h_3h_{j_1,1}h_{j_2,2}\cdots h_{j_n,n} \label{def:reduced}
\end{equation}
with $h_3\in H_3$ and $h_{j_k,k}\in C_{j_k}\setminus\{1\}$, $j_k\in\{1,2\}$ for $k=1,\ldots,n$,  such that $j_k\not= j_{k+1}$, $k=1,\ldots, n-1$, 
where $C_i\subset H_i$ is a chosen set of representatives of the right cosets of $H_3$ including the identity $1$, 
namely $H_3\times C_i\to H_i$ sending $(h,x)\to hx$ is a bijection of sets, $i=1,2$.   
\begin{lem}\label{lem:decG}
If $K=K_1\cup_{K_3}K_2$ is the union of flag complexes $K_1,K_2$ along a common flag complex $K_3$, 
then the polyhedral product $\underline{G}^K$ is the amalgamated sum 
$\underline{G}^{K_1}*_{\underline{G}^{K_3}}\underline{G}^{K_2}$.
\end{lem}
\begin{proof}
	From the presentation \eqref{def:SR}, the morphisms 
	$\underline{G}^{K_i}\to \underline{G}^K$ determined by 
	sending generators $S_i\subset G_i$ identically to $S_i\subset\underline{G}^K$ 
	are well-defined, and the diagram
	\[\begin{CD}
			\underline{G}^{K_3}@>>> \underline{G}^{K_1}\\
			@VVV                    @VVV\\
			\underline{G}^{K_2}@>>>\underline{G}^{K}
	\end{CD}
	\]
	is commutative, whence a unique morphism 
	$\alpha\co 
	\underline{G}^{K_1}*_{\underline{G}^{K_3}}\underline{G}^{K_2}\to\underline{G}^K$. To see
	that it is an isomorphism, we construct an inverse $\beta\co\underline{G}^K\to 
	\underline{G}^{K_1}*_{\underline{G}^{K_3}}\underline{G}^{K_2}$ similarly 
	by sending $S_i$ to itself. By checking the relations it is well-defined. The structure 
	theorem of amalgams, together with the presentation \eqref{def:SR}, imply that they are 
	inverse to each other.
\end{proof}

\begin{proof}[Proof of Theorem \ref{thm:normal}]
The proof is an induction on the number $m$ of vertices of $K$.
When $K$ is a simplex, $\underline{G}^K=\prod_{i=1}^m G_i$ and we have nothing to prove.  
Otherwise we decompose $K=K_1\cup_{K_3} K_2$ as a non-trivial union of flag subcomplexes which are full in $K$, by Lemma \ref{lem:L11}. 
By Lemma \ref{lem:decG}, $\underline{G}^K=\underline{G}^{K_1}*_{\underline{G}^{K_3}}\underline{G}^{K_2}$, and 
we suppose the statement holds for $\underline{G}^{K_i}$, $i=1,2,3$. We relabel the vertices if necessary, 
so that  $v_3<v_1<v_2$, for every $v_3\in V_3$ and $v_i\in V_i\setminus V_3$, $i=1,2$, 
where $V_i=\mathrm{Vert}(K_i)$ (notice that $V_3,V_2\setminus V_3,V_1\setminus V_3$ give 
a partition of $\mathrm{Vert}(K)=\{1,\ldots,m\}$). Moreover, from the proof of Lemma \ref{lem:L11}, we further assume that 
\begin{equation}
\{v_2\}=V_2\setminus V_3 \label{dec:V2}
\end{equation}
where $K_3$ is the link of $v_2$ in $K$, so that $v_2$ is not connected to any vertex from $V_1\setminus V_3$.  The induction hypothesis implies that
every word $h\in\underline{G}^K$ is in the form \eqref{def:reduced}.
Together with \eqref{dec:V2}, more explicitly we have  
\[
   h=h_3g_{v_2, 1}h_{1,2}g_{v_2,3}h_{1,4}\cdots \ \text{ or } \ h=h_3h_{1,1}g_{v_2,2}h_{1,3}g_{v_2,4}\cdots,
\]
with $g_{v_2,*}\in G_{v_2}\setminus\{1\}$, $h_{1,*}\in C_1\setminus\{1\}\subset \underline{G}^{K_1}$, where $C_1$ is a collection of right cosets of $\underline{G}^{K_3}$ in $\underline{G}^{K_1}$ in which an element $h_{1,*}$ is locally minimal, in particular, it is of the minimal word length among all elements in the orbit $\underline{G}^{K_3} \cdot h_{1,*}$.
 To show that $h$ is reduced, notice that by operations of exchanging adjacent letters, it is impossible to move a letter of the form $g\in G_{v_3}$, $v_3\in V_3$, from the middle of $h_{1,*}$ backward to the beginning, namely $h_{1,*}=g h'_{1,*}$, since otherwise its word length is not minimal in $\underline{G}^{K_3} \cdot h_{1,*}$. Therefore the only way to make $h$ shorter is to move two letters $g_i,g_j\in \underline{G}^{K_3}$ from $h_{1,i},h_{1,j}$, respectively, forward to some $h_{1,k}$ with $k\geq i,j$, and then merge them there (so that $g_i,g_j$ come from the same $G_s$ with $s\in V_3$). However, we have $i\not=j$ since each $h_{1,*}$ is already reduced; suppose $i<j$, without loss of generality, the moving of $g_i,g_j$ above implies that we can equally move $g_j$ backward to $h_{1,i}$ and merge $g_i$ there, hence at some moment we would have $h_{1,j}=g_jh'_{1,j}$, a contradiction.      
To see that $h$ is locally minimal in $\underline{G}^K$, since each word $h_{1,*}$ is already locally minimal, the only possible way to make $h$ smaller is an exchange of the last letter $x$ of $h_{1,k}$, $x\in \underline{G}^{K_3}$, and $y=g_{v_2,k+1}$, for some $k$. However, we have $xy<yx$ by induction hypothesis, namely $h$ becomes larger after this operation. As a conclusion, the induction is completed.

The uniqueness follows from the 
induction and the uniqueness of \eqref{def:reduced}. For the last statement on the multiplication 
of two elements, notice that if we call the sequence
$(i_k)_{k=1}^n$ as the type of the element $g=\prod_{i=k}^ng_{i_k,k}$, $g_{i_k,k}\in G_{i_k}$, 
the two operations will change the type, while 
the lexicographic order only depends on the type. It is easy to check that for a type of 
bounded length, we have only finitely many locally minimal types. 
\end{proof}

\subsection{Connected Hopf Algebras}
In this part we work in the category $\mathcal{H}_k^c$ 
of \emph{connected} Hopf algebras over a field $k$. For an object 
$A=A(\varphi,\Delta,\epsilon,\eta)$ in $\mathcal{H}$
we shall denote $J(A)$ as the augmentation ideal as well as the coaugmentation ideal.

To make everything explicit we specify a homogeneous basis $\mathcal{B}_A=\{c_{i}\}_{i \in\alpha}$ containing 
$c_0=1$, so that all information
is contained in the structure equations
\[
   \quad c_ic_j=\sum_{l}a_{ij}^lc_l,
\]
from the multiplication $\varphi$, together with 
\[
 \Delta c_i=\sum_{j,l}b^{jl}_ic_j\otimes c_l
\]
from the comultiplication $\Delta$, so that the coefficients $b^{jl}_i,a_{ji}^l\in k$ satisfy several equations from the axioms of a Hopf algebra. Let 
$S_A=\mathcal{B}_{A}\setminus\{1\}$ and denote by $k\langle S_A\rangle$ the free associative algebra over
$k$. Then we get the presentation 
\[
      A=\langle S_A\mid I_{A} \rangle
\]
as $k\langle S_A\rangle$ quotient by the two sided Hopf ideal $I_A$ (i.e., $\Delta I_A\subset I_A\otimes A+A\otimes I_A$), such that $\Delta$ is a morphism of algebras. 

\begin{thm}\label{thm:prepolyHopf}
Let $\underline{A}^K$ be the 
polyhedral product of Hopf algebras $\underline{A}=(A_i)_{i=1}^m$ with respect to a flag
complex $K$. Suppose that a basis $S_{A_i}$ is specified for each $A_i$, so that 
$A_i=\langle S_{A_i}\mid I_{A_i} \rangle$, $i=1,\ldots,m$. Then $\underline{A}^K$ has a presentation
\begin{equation}
 \langle \bigcup_{i=1}^m S_{A_i}\mid I_{A_1},\ldots, I_{A_m}, [S_{A_i},S_{A_j}], \forall\{i,j\}\in K\rangle
 \label{def:A'}
\end{equation}
where $[S_{A_i},S_{A_j}]$ is the two-sided ideal generated by $[s_i,s_j]=s_is_j-(-1)^{|s_i||s_j|}s_js_i$, with $s_i,s_j$
running through all elements from $S_{A_i},S_{A_j}$, respectively. $\underline{A}^K$ is cocommutative if each
$A_i$ is.
\end{thm}
\begin{proof}
Let $A$ be the algebra given by \eqref{def:A'}. It can be checked directly that the diagonal 
$\Delta\co A\to A\otimes A$ satisfies the comultiplication  axiom  in $A$, which is given by
\begin{equation}
 \Delta (\varphi(c_{i_1,1}\otimes c_{i_2,2}\otimes \cdots \otimes c_{i_n,n}))= 
 (\varphi\otimes\varphi)\circ T(\bigotimes_{k=1}^n \Delta c_{i_k,k})\label{def:diag}
\end{equation}
with $\varphi\co k\langle \bigcup_{i=1}^m S_{A_i}\rangle\to A$ 
the multiplication in $A$ and 
\[
   T\co
   \bigotimes_{k=1}^nA_{i_k}\otimes A_{i_k}\to \otimes_{k=1}^nA_{i_k}\bigotimes \otimes_{k=1}^nA_{i_k}
\]
the shuffle map sending a homogeneous element $\otimes_{k=1}^n(s'_{i_k}\otimes s''_{i_k})$ to 
$(-1)^{\nu}\otimes_{k=1}^n s'_{i_k}\otimes\otimes_{k=1}^n s''_{i_k} $, where $\nu=\sum_{k=1}^n \deg s''_{i_k}
\sum_{k'>k}\deg s'_{i_{k'}}$. It is straightforward to check that 
$A$ is cocommutative, if each $A_i$ is.

To see that $A$ is a Hopf algebra, it suffices to 
check the two-sided ideal $I_{ij}$ generated by $[S_{A_i},S_{A_j}]$ satisfies 
\begin{equation}
\Delta I_{ij}\subset I_{ij}\otimes A+A\otimes I_{ij}.\label{ideal:Hopf}
\end{equation} 
We write $\Delta (s_i)=\sum_{l,t}b_i^{lt}c_l\otimes  c_t\subset A_i\otimes A_i$ 
and $\Delta(s_j)=\sum_{l',t'}b_j^{l't'}c_{l'}'\otimes  c_{t'}'\in A_j\otimes A_j$.
Then 
\begin{align*}\Delta ([s_i,s_j])&=\sum_{l,l',t,t'}b_i^{lt}b_j^{l't'}(-1)^{|c_{l'}||c_t|}\left(c_lc_{l'}\otimes  c_tc_{t'}-(-1)^{|c_{l}||c_{l'}|+|c_{t}||c_{t'}|}c_{l'}c_{l}\otimes  c_{t'}c_{t}\right)\\
&=\sum_{l,l',t,t'}b_i^{lt}b_j^{l't'}(-1)^{|c_{l'}||c_t|}\left([c_l,c_{l'}]\otimes  c_tc_{t'}+(-1)^{|c_{l}||c_{l'}|}c_{l'}c_{l}\otimes [c_{t},c_{t'}]\right),
\end{align*}
hence \eqref{ideal:Hopf} holds. We see that $A$ is a well-defined Hopf algebra and it remains 
to check the universal property as a polyhedral product. 
First we define 
$\iota_{\sigma}\co   \bigotimes_{i\in\sigma}A_{i}\to A $ with 
$\iota_{\sigma}(\otimes_{i\in\sigma}c_i)=\varphi(\otimes_{i\in\sigma}c_i)$. Notice that by definition,
\[\bigotimes_{i\in\sigma}A_{i}=\langle \cup_{i\in\sigma} S_i\mid I_{A_i}, [S_i,S_j],\ \forall\{i,j\}\subset\sigma\rangle,\]
hence $\iota_{\sigma}$ is well-defined, and the diagram 
\[\xymatrix{
	    \bigotimes_{i'\in\sigma'}A_{i'} \ar[rr]^{F(\sigma'\to\sigma)} \ar[rrd]^{\iota_{\sigma'}} &  &   \bigotimes_{i\in\sigma}A_{i}\ar[d]^{\iota_{\sigma}}\\
	                   &        &                A
	}
	\]
commutes for every pair $(\sigma,\sigma')$ of simplices from $K$, in which 
$F(\sigma'\to\sigma)$ sends  $\bigotimes_{i'\in\sigma'}A_{i'}$ identically onto its image $\bigotimes_{i\in\sigma}H_{i}$, 
	where $H_{i'}=A_{i'}$ if $i'\in\sigma'$, otherwise $H_{i}=k\langle 1\rangle$ is the trivial Hopf algebra. 
	
	Now suppose $A'$ is any object in $\mathcal{H}_k^c$, 
with morphisms $f_{\sigma}\co \bigotimes_{i\in\sigma}A_{i}\to A'$ such that
\[\xymatrix{
	    \bigotimes_{i'\in\sigma'}A_{i'} \ar[rr]^{F(\sigma'\to\sigma)} \ar[rrd]^{f_{\sigma'}} &  &   \bigotimes_{i\in\sigma}A_{i}\ar[d]^{f_{\sigma}}\\
	                   &        &                A'
	}
	\]
	commutes for every pair $(\sigma,\sigma')$. We define $u\co A\to A'$ so that $u(s_i)=f_i(s_i)$ for all $s_i\in S_i\subset A_i$. In this way $u$ is uniquely determined and the universal property holds, once we show that
	it is well-defined, namely $u([S_i,S_j])=1$ for all $\{i,j\}\in K$. 
	This is clear from the two commutative diagrams above.
\end{proof}
\begin{remark}
	Given morphisms $A_i\to \wt{A}_i$ in $\mathcal{H}_k^c$, 
	$i=1,\ldots,m$, together with a simplicial 
inclusion $K\to \wt{K}$ of flag complexes with the same vertices, the universal property induces a 
unique morphism $\underline{A}^K\to\underline{\wt{A}}^{\wt{K}}$. With the presentations \eqref{def:A'}, 
this morphism sends $[S_i,S_j]$ to $0\in \underline{\wt{A}}^{\wt{K}}$, if $\{i,j\}$ spans an edge in $\wt{K}$.
\end{remark}
\begin{lem}\label{lem:decA}
If $K=K_1\cup_{K_3}K_2$ is the union of flag complexes $K_1,K_2$ along  a common flag complex $K_3$, 
then the polyhedral product $\underline{A}^K$ is isomorphic to the colimit
$\underline{A}^{K_1}\coprod_{\underline{A}^{K_3}}\underline{A}^{K_2}$.
  \end{lem}
\begin{proof}
Notice that on the set of vertices we have $\mathrm{Vert}(K_3)=\mathrm{Vert}(K_1)\cap\mathrm{Vert}(K_2)$. 
The presentation \eqref{def:A'} implies that $\underline{A}^{K_3}\to\underline{A}^{K_i}$, $i=1,2,$, are well-defined. To 
see that they are monomorphisms, consider the projections $\underline{A}^{K_i}\to\underline{A}^{K_3}$ sending $s_j=0$, 
for all $s_j\in S_j$ with $\{j\}\not\in K_3$. Again the presentation \eqref{def:A'} implies that these projections are well-defined, each giving a left inverse to $\underline{A}^{K_3}\to\underline{A}^{K_i}$. Finally the universal property implies
a morphism 
$\underline{A}^{K_1}\coprod_{\underline{A}^{K_3}}\underline{A}^{K_2}\to\underline{A}^K$, 
which is an isomorphism by constructing an inverse that is identical on $S_i$, as in the proof
of Lemma \ref{lem:decG}. 
\end{proof}

\begin{defin}\label{def:order2}
	A basis element  of $A$ of the form
	$a=\prod_{k=1}^ns_{i_k,k}$ with letters $s_{i_k,k}\in S_{i_k}$ is \emph{reduced}
(the identity $1$ is the empty word), 
if the two operations of I) 
exchanging $s_{i_k,k}s_{i_{k+1},k+1}$ into $s_{i_{k+1},k+1}s_{i_k,k}$, when 
$\{i_k,i_{k+1}\}\in K$ and II) merging $s_{i_{k},k}s_{i_{k+1},k+1}$ into a linear sum
$\sum_tk_tg^t_{i_{k},k}$ with $k_t\in k$ and $g^t_{i_{k},k}$ a product of basis elements from $S_{i_k}$ (now $i_k=i_{k+1}$), will not make $a$ into a linear sum of product of basis elements of shorter length. 
The \emph{lexicographic partial order} on reduced words 
	is given by rules 
	1)$0<1<s_{1}<s_{2}<\ldots<s_{m}$ for every $s_i\in S_i$, 
and 2) two reduced element
	$a=\prod_{k=1}^ns_{i_k,k}<a'=\prod_{k=1}^{n'}s'_{i_{k}',k}$, 
	if there exists an index $j\leq n$, 
	such that $s_{i_k,k}=s'_{i_{k}',k}\in S_{i_k}$ 
	for $k\leq j$ while $i_{j+1}<i_{j+1}'$ (if $j=n$, this means
	 $n'>n$ and the first $n$ letters of $a'$ coincide with that of $a$).
	
	A reduced element 
is \emph{locally minimal} if a single
operation I), whenever possible,  
will make it larger in the partial order above. 
\end{defin}
 
  \begin{thm}\label{thm:Hopfmain}
  The polyhedral product $\underline{A}^K$ associated to $m$ objects $\underline{A}=(A_i)_{i=1}^m$
  in $\mathcal{H}_k^c$ and a flag complex $K$ has a linear basis consisting of all locally 
  minimal elements of the form
   $a=\prod_{k=1}^ns_{i_k,k}$, with $s_{i_k,k}\in S_{i_k}$.
  \end{thm}
  We recall some facts before the proof of this theorem.
  Let $ A_3\to A_i$, $i=1,2$, be two monomorphisms of $\mathcal{H}_k^c$.
	The colimit $A=A_1\coprod_{A_3}A_2$ of the diagram $A_1\leftarrow A_3\rightarrow A_2$ is 
	well-defined (see \cite{MS68}), and 
        \[
		C_i=k\otimes_{A_3}A_i 
	\]
	inherits the structure from $A_i$ to be a coalgebra, $i=1,2$.
	Since $k$ is a field, we can choose a section of $C_i$ in $A_i$, and a theorem
	of Milnor and Moore \cite{MM65} states that there exist isomorphisms  
	\begin{equation}
		A_i\cong A_3\otimes C_i \label{eq:mm}
     \end{equation}
     as free left $A_3$-modules and right $C_i$ comodules. Now let $\wt{C}_i$ be the cokernel of 
     $k\to C_i$, $i=1,2$, and we define 
	two sequences of modules $\{M_n^i\}_{n=1}^{\infty}$ inductively: 
$M_1^2=\wt{C}_1$, $M_2^2=\wt{C}_1\otimes \wt{C}_2$, 
\[M_{n}^2=\begin{cases}M_{n-1}^2\otimes \wt{C}_{1}=\underbrace{\wt{C}_1\otimes \wt{C}_2\otimes 
		\wt{C}_1\otimes\cdots\otimes \wt{C}_1}_{n} 
			& \text{$n$ odd}\\
			M_{n-1}^2\otimes \wt{C}_{2}=\underbrace{\wt{C}_1\otimes \wt{C}_2\otimes 
			\wt{C}_1\otimes\cdots\otimes \wt{C}_2}_{n}     &     \text{$n$ even},
		\end{cases}
	\]
	similarly $M_1^1=\wt{C}_2$ and
	\[M_{n}^1=\begin{cases}M_{n-1}^1\otimes \wt{C}_{2} & \text{$n$ odd}\\
			M_{n-1}^1\otimes \wt{C}_{1}     &     \text{$n$ even},
		\end{cases}
	\]
	then it follows from the  structure theorem of amalgams of connected Hopf algebras (see \cite[Theorem 5.1.4, pp. 105--107]{Lem74})
	that we have an isomorphism
	\begin{equation}
		A\cong A_3\otimes \left(k\oplus\bigoplus_{n=1}^{\infty}M_n^1\oplus\bigoplus_{n=1}^{\infty}M_n^2
		\right) \label{iso:A1}
\end{equation}
 of left $A_3$-modules, as well as isomorphisms
 \begin{equation}
	 A\cong A_i\otimes \left(k\oplus\bigoplus_{n=1}^{\infty}M_n^i\right) \label{iso:Ai}
 \end{equation}
 of left $A_i$-modules, $i=1,2$. The following theorem is a conclusion.
 \begin{thm}[Lemaire-Serre\footnote{In the proof of this theorem Lemaire cited Serre's work as \emph{Cours Coll. France 1968-1969 (multigraphi\'{e})} (see \cite[pp. 105--107]{Lem74}).}]\label{thm:LS}
	 Let $ A_3\to A_i$, $i=1,2$ be two monomorphisms in $\mathcal{H}_k^c$. 
	 Suppose we specify a homogeneous basis $\mathcal{B}_3$ which spans $A_3$, as well as two sets 
	 $Z_i\subset A_i$ such that 
	 the canonical epimorphisms $A_i\to C_i=k\otimes_{A_3}A_i\to\wt{C}_i$ 
	 send $Z_i$ one-to-one onto a basis of $\wt{C}_i=\mathrm{coker}(k\to C_i)$, $i=1,2$. 
 Then the colimit $A$ of the diagram $A_1\leftarrow A_3\rightarrow A_2$ has a basis of the 
 form 
 \begin{equation}
          a_3c_{j_1,1}c_{j_2,2}\cdots c_{j_l,l}\label{def:left}
 \end{equation}
 with $a_3$ running through all elements in $\mathcal{B}_3$ and $c_{j_k,k}$ running through $Z_{j_k}$, such that $j_k\not=j_{k+1}$.
 
 Similarly $A$ has a basis of the form 
 \begin{equation}
 c_{j_l,l}\cdots c_{j_2,2}c_{j_1,1}a_3\label{def:right}
 \end{equation}
 if $C_i=A_i\otimes_{A_3}k$, $i=1,2$.
     \end{thm}
\begin{proof}[Proof of Theorem \ref{thm:Hopfmain}]
The proof is again an induction on the number $m$ of vertices of $K$.
When $K$ is a simplex, $\underline{A}^K=\bigotimes_{i=1}^m A_i$. Clearly locally minimal elements 
$a=\prod_{i=1}^ns_{i_1,1}\ldots s_{i_n,n}$ with $1\leq i_1<i_2<\ldots<i_n\leq m$ give a basis.  
Otherwise we decompose $K=K_1\cup_{K_3} K_2$ as a non-trivial union of flag subcomplexes which are full in $K$, by Lemma \ref{lem:L11}. 
We have  $\underline{A}^K=\underline{A}^{K_1}\coprod_{\underline{A}^{K_3}}\underline{A}^{K_2}$ by 
Lemma \ref{lem:decA}, and 
suppose the statement already holds for $\underline{A}^{K_i}$, $i=1,2,3$. We relabel the vertices if necessary, 
so that  $v_3<v_1<v_2$, for every $v_3\in V_3$, $v_i\in V_i\setminus V_3$, $i=1,2$, 
where $V_i=\mathrm{Vert}(K_i)$. The induction hypothesis implies that
a basis of the quotient $C_i=k\otimes_{\underline{A}^{K_3}}\underline{A}^{K_i} $ 
can be chosen as locally minimal elements, in particular, each basis element has the minimal word length among all representatives in the same equivalent class. 
 Moreover, using a similar argument as in the proof of Theorem \ref{thm:normal}, it can be checked that an element of the form
\begin{equation}
       a=a_3c_{j_1,1}c_{j_2,2}\ldots c_{j_l,l},\label{def:a}
\end{equation}
with $a_3$ a locally minimal in $\underline{A}^{K_3}$ and $c_{j_k,k}$ locally minimal in  
$\wt{C}_{j_k}$ where $j_k\in \{1,2\}$, $k=1,\ldots,l$, is again locally minimal in $\underline{A}^K$. Finally, the proof is completed by
the induction, together with Theorem \ref{thm:LS}.  
\end{proof}

\end{document}